\documentclass[11pt]{amsart}
\pdfoutput=1
\usepackage{amssymb}
\usepackage{amsthm}
\usepackage{amsmath}
\usepackage{hyperref}
\usepackage{manfnt}
\usepackage{xfrac}

\usepackage[protrusion=true,expansion=true]{microtype}

\usepackage{geometry}
\theoremstyle{plain} % default
\newtheorem{theorem}[equation]{Theorem}

\newtheorem{proposition}[equation]{Proposition}

\newtheorem{corollary}[equation]{Corollary}
\theoremstyle{definition}
\newtheorem{definition}[equation]{Definition}

\theoremstyle{remark}
\newtheorem{remark}[equation]{Remark}
\numberwithin{equation}{section}

\makeatletter

\copyrightinfo{}{The authors}

\def\speciallabelmark#1{\def\@currentlabel{#1}}

\def\makeoverbar#1#2#3#4#5#6#7{%
 \setbox0=\hbox{$\m@th#2\mkern#5mu{{}#3{}}\mkern#6mu$}%
 \setbox1=\null \dimen@=#4\fontdimen8#13 \dimen@=3.5\dimen@ 
 \advance\dimen@ by \ht0 \dimen@=-#7\dimen@ \advance\dimen@ by \wd0
 \ht1=\ht0 \dp1=\dp0 \wd1=\dimen@
 \dimen@=\fontdimen8#13 \fontdimen8#13=#4\fontdimen8#13
 \rlap{\hbox to \wd0{$\m@th\hss#2{\overline{\box1}}\mkern#5mu$}}
 \fontdimen8#13=\dimen@}

\makeatother

\def\0{\kern0pt\-\nobreak\hskip0pt\relax}

\newcommand{\re}{{\mathbb R}}    % real numbers
\newcommand{\na}{{\mathbb N}}    % natural numbers
\newcommand{\cn}{{\mathbb C}}
\newcommand{\qn}{{\mathbb H}}    % complex numbers
\newcommand{\E}{{\mathbb E}}

\newcommand{\5}{\mskip-1.5mu}

\newcommand{\expv}[1]{\operatorname{\E}\5\5\left[ #1 \right]}
\begin{document}

\title[Fine asymptotics for models with Gamma type moments]{Fine asymptotics for models with Gamma type moments}

\author[P.~Eichelsbacher, L. Knichel]{Peter Eichelsbacher and Lukas Knichel}
\address[P.~Eichelsbacher]{Faculty of Mathematics, NA 3/66, Ruhr-Universit\"at Bochum, 44780 Bochum, Germany}
\email{\href{mailto:Peter.Eichelsbacher@ruhr-uni-bochum.de}{Peter.Eichelsbacher@ruhr-uni-bochum.de}}
\address[L.~Knichel]{Faculty of Mathematics, NA-S\"ud EG-7, Ruhr-Universit\"at Bochum, 44780 Bochum, Germany}
\email{\href{mailto:Lukas.Knichel@ruhr-uni-bochum.de}{Lukas.Knichel@ruhr-uni-bochum.de}}

%\urladdr{\href{http://www.ruhr-uni-bochum.de/ffm/Lehrstuehle/stochastik/eichelsbacher.html}
%{http://www.ruhr-uni-bochum.de/ffm/Lehrstuehle/stochastik/eichelsbacher.html}}

%\thanks{$^*$}

\subjclass{\href{http://www.ams.org/mathscinet/msc/msc2010.html?t=60Fxx}{15B52, 15A15, 52A22, 52A23, 60B20, 60D05, 60F05, 60F10, 62H10}}
\renewcommand{\subjclassname}{\textup{2010} Mathematics Subject
     Classification}


\begin{abstract}
The aim of this paper is to give fine asymptotics for random variables with moments of
Gamma type. Among the examples we consider are random determinants of Laguerre and Jacobi beta ensembles with varying 
dimensions (the number of observed variables and the number of measurements vary and may be different).
In addition to the Dyson threefold way of classical random matrix models (GOE, GUE, GSE), we study random determinants of random matrices of the so-called tenfold way,  
including the Bogoliubov-de Gennes and chiral ensembles from mesoscopic physics. We show that fixed-trace matrix ensembles
can be analysed as well.
Finally, we add fine asymptotics for the $p(n)$-dimensional volume of the simplex with $p(n)+1$ points in $\re^n$ distributed according to special distributions, which
is strongly correlated to Gram matrix ensembles.
We use the framework of mod-$\varphi$ convergence to obtain extended limit theorems, Berry-Esseen bounds, precise moderate deviations, large and moderate
deviation principles as well as local limit theorems. The work is especially based on the recent work of Dal Borgo, Hovhannisyan and Rouault \cite{Borgoetal:2017}.
\end{abstract}

\maketitle

\section{Introduction}

\noindent
{\it Moments of Gamma type:}
In the excellent survey \cite{Janson:2010} (see also \cite{further}) a positive random variable $X$ is defined to have {\it moments of Gamma type} if, for $s$ in some interval,
$$
\expv{X^s} = C D^s \frac{\prod_{j=1}^J \Gamma(a_j s + b_j)}{\prod_{k=1}^K \Gamma(a'_k s + b'_k)}
$$

for some integers $J,K \geq 0$ and some real constants $C, D >0$, $a_j, b_j, a'_k, b'_k$.
In \cite{Janson:2010} and \cite{further} a rich class of examples was presented. This includes many standard distributions, among them the Gamma distribution, 
the Beta distribution, stable distributions, the Mittag-Leffler distribution,
extreme value distributions, the Fej\'er distribution, and many more distributions, see \cite{Janson:2010}. Examples from the point of view of stochastic models can be found in \cite{further}, where blocks in a Stirling permutation, distances in a sphere, preferential attachment random graphs, the maximum of i.i.d. exponentials, the largest values of i.i.d. exponentials as well as discriminants and Selberg's integral formula in random matrix theory are studied. 

For our purposes we will choose $J=K=p$, which might
depend on a parameter, say $p(n)$. Moreover, we choose $b_j = b'_j = \alpha(j+l)$ for some $l$ which might depend on $n$ and $p(n)$
and a constant $\alpha$, and we will choose $a'_j=0$.

Although there is a rich family of examples of positive random variables with moments of Gamma type, we will mainly be motivated by determinants of random matrices as well as by volumes of random parallelotopes and random simplices. 

\medskip

\noindent
{\it Mod-Gaussian convergence:}
Mod-Gaussian convergence was introduced in \cite{JacodKowalskiN:2011} and was mainly inspired by theorems and conjectures
in random matrix theory and number theory concerning moments of values of characteristic polynomials or zeta functions.
One of the canonical motivating examples is due to \cite{KeatingSnaith:2000}. The authors proved that if $(Y_n)_n$ is a sequence of complex random variables with $Y_n$ distributed like $\det (I - X_n)$ for some random variable $X_n$ taking values in the unitary group
$U(n)$ and uniformly distributed on $U(n)$, i.e. distributed according to the Haar measure (called the circular unitary ensemble - CUE), then for any complex number $z$ with
${\operatorname{Re}}(z) > -1$ we have
\begin{equation} \label{Mellinfirst}
\E \bigl( |Y_n|^z \bigr) = \prod_{k=1}^n \frac{\Gamma (k) \Gamma(k+z)}{(\Gamma (k + \frac z2))^2},
\end{equation}
which is of Gamma type. A consequence of \eqref{Mellinfirst} is that the value distribution of the real and the imaginary part of $\log Y_n / \sqrt{ \frac 12 \log n}$ converge independently to the standard Gaussian law. Here one has to clarify the choice of the branch of logarithm, see \cite[par. after(7)]{KeatingSnaith:2000}. 
The means to prove this central limit theorem is the method of cumulants, showing that
the $k$'th cumulant is converging to zero for all $k \geq 3$. Bounding the cumulants, finer asymptotics like Cram\'er-type moderate deviations, Berry-Esseen bounds and moderate deviation principles have been considered in \cite{DoeringEichelsbacher:2010}. 
An important contribution of \cite[see (15)]{KeatingSnaith:2000} is that the authors obtained that for any complex number $z$ with ${\operatorname{Re}}(z) > -1$
$$
\lim_{n \to \infty} \frac{1}{n^{z^2}} \E \bigl( |Y_n|^{2z} \bigr) = \lim_{n \to \infty} 
\frac{1}{e^{z^2 \log n}} \E \bigl( e^{ 2 z \log |Y_n|} \bigr) = \frac{G(1+z)^2}{G(1 + 2 z)},
$$
where $G$ is the Barnes (double gamma) function, see Appendix. This is {\it mod-Gaussian convergence} in a set $D \subset \cn$ that contains 0, with parameters $t_n= \frac 12 \log n$ and limiting function $\frac{G( 1 + \frac z2)}{G(1+z)}$, see Definition \ref{moddef}. This renormalised convergence of the moment generating function was not standard in probability theory before
\cite{KeatingSnaith:2000} and  \cite{JacodKowalskiN:2011}. In \cite{KeatingSnaith:2000}, the authors also considered the circular orthogonal and the circular symplectic ensembles (COE, CSE). For the circular ensembles in \cite[Section 7.5]{modbook} and \cite[Section 3.5]{mod2}, more information encoded in mod-Gaussian convergence is discovered. 

The proof of \eqref{Mellinfirst} is nowadays considered standard. In the CUE average, one starts with the joint distribution of the eigenphases
due to Weyl (\cite{Weyl:1946}, see also \cite{Mehta:book}) and apply Selberg's integral formulas, see Chapter 17 in \cite{Mehta:book}.
For many other random matrix ensembles the joint distribution of eigenvalues or eigenphases can be calculated 
explicitly as well, and Selberg-type formulas are available, see \cite{Mehta:book} and \cite{Hiai/Petz:Semicircle}.
\medskip

\noindent
{\it Random determinants:}
In general, the distribution of random determinants is an important functional in random matrix theory.  Moreover, random determinants are constantly used in random geometry to compute volumes of parallelotopes (\cite{Mathai:1999}) and in multivariate statistics to build tests.

In a series of papers, moments of the determinants were computed, see \cite{Prekopa:1967} and \cite{Dembo:1989} and references therein. In \cite{TaoVu:2006}, it was proved for Bernoulli random matrices, that with probability tending to one as $n$ tends to infinity
\begin{equation} \label{LLNiidAn}
\sqrt{n!} \exp(-c \sqrt{n \log n}) \leq | \det A_n | \leq \sqrt{n!} \, \, \omega(n)
\end{equation}
for any function $\omega(n)$ tending to infinity with $n$. This shows that almost surely, $\log | \det A_n|$ is $( \frac 12 + o(1)) n \, \log n$. \cite{Goodman:1963} considered the random Gaussian case, where the entries of $A_n$ are iid standard real Gaussian variables. Here, the square of the determinant can be expressed as a product of independent chi-square variables and it was proved that
\begin{equation} \label{CLTiidAn}
\frac{ \log (|\det A_n|) - \frac 12 \log n! + \frac 12 \log n }{ \sqrt{ \frac 12 \log n}} \to N(0,1)_{\re},
\end{equation}
where $N(0,1)_{\re}$ denotes the real standard Gaussian (convergence in distribution). A similar analysis also works for complex Gaussian matrices.
\cite{Girko:1979} stated that \eqref{CLTiidAn} holds for real iid matrices under the assumption that the fourth moment of the atom variables is 3. In \cite{Girko:1998} the author claimed the same result when the atom variables have bounded $(4 + \delta)$-th moment. 
Recently, \cite{NguyenVu:2011} gave a proof for \eqref{CLTiidAn} under an exponential decay assumption on the entries. They also present an estimate for the rate of convergence, which is that
the Kolmogorov distance of the distribution of the left hand side of \eqref{CLTiidAn} and the standard real Gaussian can be bounded by $(\log n)^{- \frac 13 +o(1)}$. In the Gaussian case this was improved in \cite{DoeringEichelsbacher:2012b} to $(\log n)^{-\frac 12}$. 
The analogue of \eqref{LLNiidAn}
for Hermitian random matrices was first proved in \cite[Theorem 31]{Tao/Vu:2009} as a consequence of the famous Four Moment Theorem. Even in the Gaussian case, it is not simple to prove an analogue of the central limit theorem (CLT) \eqref{CLTiidAn}. The observations in \cite{Goodman:1963} do not apply due to the dependence between the rows. In \cite{MehtaNormand:1998} and in \cite{DelannayLeCaer:2000}, the authors computed the moment generating function of the log-determinant for the Gaussian unitary and Gaussian orthogonal ensembles respectively,  and discussed the central limit theorem via the method of cumulants (see \cite[equation (40) and Appendix D]{DelannayLeCaer:2000}):
consider a Hermitian $n \times n$ matrix $X_n$ in which the atom distribution $\zeta_{ij}$ are given by the complex Gaussian $N(0,1)_{\cn}$ for $i < j$ and the real Gaussian $N(0,1)_{\re}$ for $i=j$ (which is called the Gaussian Unitary Ensemble (GUE), see Section 2.4). The calculations in  \cite{DelannayLeCaer:2000} 
imply a central limit theorem:
\begin{equation} \label{CLTiidMnC}
\frac{ \log (|\det X_n|) - \frac 12 \log n! + \frac 14 \log n }{ \sqrt{ \frac 12 \log n}} \to N(0,1)_{\re},
\end{equation}
Recently, \cite{TaoVu:2011} presented a different approach to prove this result approximating the $\log$-determinant as a sum of weakly dependent
terms, based on analysing a tridiagonal form of the GUE due to Trotter \cite{Trotter:1984}. They have to apply stochastic calculus and a martingale central limit theorem to get their result. 

In \cite{DoeringEichelsbacher:2012b}, Cram\'er-type moderate deviations and Berry-Esseen bounds for the log-determinant of the GUE and GOE ensembles were considered, establishing good bounds
for all cumulants. 

In \cite[Section 2.2]{Rouault:2007}, moment formulas comparable to \eqref{Mellinfirst} 
were collected for the product of eigenvalues (determinants) for certain families of random matrices
like $\beta$-Laguerre-, $\beta$-Jacobi- and uniform Gram ensembles. 
Laguerre ensembles are presenting sample covariance matrices in mathematical statistics, Jacobi ensembles are connected to correlation coefficients.
The role of Gram ensembles is the following.
\medskip

\noindent
{\it Gram matrices in random geometry:}
The uniform Gram matrix is connected to the volume or $p(n)$-content of a {\it parallelotope}. Here the model is to consider
random vectors $(b_i)_i$ with the same distribution $\nu_n$ in $\re^n$. If one takes the Gaussian law $\nu_n = N(0,I_n)$,
then the distribution of $B^t \, B$ with $B$ the $p(n) \times n$ matrix with $p(n)$ column vectors $b_1, \ldots, b_{p(n)}$ of $\re^n$, 
is the so-called {\it Wishart ensemble}. This ensemble is also called the $1$-Laguerre ensemble because of the connection with orthogonal polynomials. If $\nu_n$ is the uniform distribution on the unit sphere in $\re^n$, then the matrix ensemble is called the uniform Gram ensemble. The joint density of the non-diagonal entries
can be presented explicitly, see \cite[(2.1)]{Rouault:2007}. See also \cite{Mathai:1997} and \cite{Mathai:1999} for
Gram matrices in random geometry with further distributions $\nu_n$.

Formulas for moments of certain random matrix ensembles in \cite{Rouault:2007} imply precise formulas for the distribution
of the corresponding determinants, see \cite[Proposition 2.2, 2.3, 2.4]{Rouault:2007}. These results correspond to Bartlett-type results (see Proposition 2.1 in \cite{Rouault:2007}). Among others, asymptotic log-normality was proved for these ensembles, see \cite[Theorem 3.2, 3.5, 3.8]{Rouault:2007}. 
\medskip

\noindent
{\it Second-order analysis:}
Recently, in \cite{Borgoetal:2017}, second-order refinements of central limit theorems for log-determinants of certain random matrix ensembles 
were considered. The authors provide an asymptotic expansion of the Laplace transforms of the log-determinants and apply
the framework of mod-Gaussian convergence.  Their results include mod-Gaussian convergence,
extended central limit theorems, precise moderate deviations, Berry-Esseen bounds as well as local limit theorems. 
Moreover, they were able to apply the techniques to random characteristic polynomials evaluated at 1 for circular and
circular Jacobi beta ensembles.
\medskip

\noindent
The aim of this paper is to study precise asymptotics (second-order analysis) for determinants of $\beta$-Laguerre-ensembles for $p(n) \times p(n)$ random matrices $A^{\dagger} A$, where $A$ is a certain $p(n) \times n$ matrix and $A^{\dagger}$ denote the transpose, the Hermitian conjugate or the dual of $A$ when $A$ is real, complex and quaternion respectively. 
Since the interpretation of $p(n)$ is the number
of variables and $n$ is the number of observations, we are mainly interested in the case $p(n) \not= n$. The case $p(n)=n$ was studied
in \cite{Borgoetal:2017}. We will observe that mod-Gaussian convergence can be proved for $n-p(n)$ being small,
or being fixed, or is allowed to grow to infinity at a certain rate. Moreover we observe a mod-stable convergence
on the imaginary line $i \re$ if the number of variables $p$ is fixed. We will present these results in Theorem \ref{cumulant} and Theorem \ref{theoremmodL}.
Corresponding pieces of information encoded in the mod-convergence
will be worked out, among them precise moderate deviations, moderate and large deviation principles, Berry-Esseen
bounds and local limit theorems for log-volumes or log-determinants. An important observation of our paper is that
the asymptotic behaviour of the determinants of $\beta$-Laguerre ensembles for varying dimensions {\it is sufficient} to be able to study the asymptotics of determinants
of $\beta$-Jacobi ensembles, of Ginibre ensembles, of 7 further matrix ensembles 
within the tenfold way of mesoscopic physics, and of the determinant of certain Gram matrices with respect to certain distributions in $\re^n$, representing the volume of parallelotopes.

\noindent
The structure of the paper goes as follows. In the next section we will collect all random matrix
ensembles we are interested in. In Section 3 we recall the tool of mod-$\phi$ convergence and the limiting theorems which are implied by this notion.
In Section 4 we present a key asymptotic, following Theorem 5.1 in \cite{Borgoetal:2017}. In Section 5, the key asymptotic of Section 4 is mainly applied to the case of a $\beta$-Laguerre ensemble. The main result is Theorem \ref{cumulant}. 
In Section 6, we prove mod-$\phi$ convergence, extended limit theorems, precise deviations, large and moderate deviations and Berry-Esseen bounds for log determinants of our collection of random matrix ensembles. Section 7 contains the results for
random parallelotopes and random simplices in high dimensions.

\section{Classes of models}
In this section we will collect all classes of examples we are interested in.

\subsection{Wishart matrices / Laguerre ensembles}
As our first motivating example let us consider the following prototype of a random matrix ensemble from mathematical statistics. 
The study of sample covariance matrices is fundamental in multivariate statistics. Typically, one thinks of $p(n)$ variables $y_k$ with each variable
measured or observed $n$ times. One is interested in analysing the covariance matrix $A^t \, A$, with $A$ being the $n \times p(n)$ matrix with $p(n) \leq n$, and entries $y_k^{(j)}$ for $j=1, \ldots, n$ and $k=1,\ldots, p(n)$. If $A$ is chosen to be a Gaussian matrix
over $\re$, $\cn$ or $\qn$, the distribution of the $p(n) \times p(n)$ random matrix $A^{\dagger} A$ is called {\it Laguerre}
real, complex or symplectic ensemble. Here $A^{\dagger}$ denotes the transpose, the Hermitian conjugate or the dual of $A$
accordingly, when $A$ is real, complex or quaternion. The eigenvalues $(\lambda_1, \ldots, \lambda_{p(n)})$ are real and non-negative and it is a well known fact that the joint density function on the set $(0, \infty)^{p(n)}$ is 
$$
\frac{1}{Z_{n,p(n),\beta}} \prod_{1 \leq j < k \leq p(n)} |\lambda_j - \lambda_k|^{\beta} \prod_{k=1}^{p(n)} \bigl( \lambda_k^{\frac{\beta}{2} (n - p(n) +1) - 1} e^{-\frac{\lambda_k}{2}} \bigr)
$$
for $\beta =1,2,4$ respectively, see for example \cite[Proposition 3.2.2]{Forrester:book}. Tridiagonal models for the $\beta$-Laguerre Ensembles have been constructed in \cite{DumitriuEdelman:2002}. Using Selberg integration
from \cite[(17.6.5)]{Mehta:book}, we obtain
$$
Z_{n,p(n),\beta} = 2^{\frac{\beta}{2} n p(n) - p(n)} \prod_{k=1}^{p(n)} \frac{\Gamma(1+\frac{\beta}{2} k) \Gamma( \frac{\beta}{2}(n -p(n)) + \frac{\beta}{2} k)}{\Gamma(1 + \frac{\beta}{2})}.
$$
Using this Selberg formula, one obtains directly that
\begin{eqnarray*}
\E \biggl[ \biggl( \det W_{n,p(n)}^{L, \beta} \biggr)^z \biggr] & = &2^{p(n) z} \prod_{k=1}^{p(n)} \frac{\Gamma\bigl(
\frac{\beta}{2} (n-p(n)+k) +z \bigr)}{\Gamma\bigl(\frac{\beta}{2} (n -p(n) +k)\bigr)} \\
&=& 2^{p(n) z} \prod_{k=1+n-p(n)}^{n} 
\frac{\Gamma\bigl(\frac{\beta}{2} k +z \bigr)}{\Gamma\bigl(\frac{\beta}{2} k\bigr)},
\end{eqnarray*}
where $W_{n,p(n)}^{L, \beta}$ denotes the $\beta$-Laguerre distributed random matrix of dimension  $p(n) \times p(n)$. This object is called the {\it Mellin transform}
of the determinant, which is defined for any $z \in \cn$ with ${\operatorname{Re}}(z) > -\frac{\beta}{2}$.

\noindent
We introduce the notion
\begin{equation} \label{L}
L(p,l,\alpha;z) = \log \biggl( \prod_{k=1}^p \frac{\Gamma(\alpha(k+l) + z)}{\Gamma( \alpha(k+l))} \biggr) ,
\end{equation}
with $p, l \geq 1$ and $z \in \cn$ with $\operatorname{Re}(z) > - \alpha$ and $\alpha \in \re$ and obtain
\begin{equation} \label{MellinL}
\log \E \biggl[ \biggl( \det W_{n,p(n)}^{L, \beta} \biggr)^z \biggr] = z p(n) \log 2 + L(p(n), n-p(n), \beta/2;z).
\end{equation}
In the case $p(n)=n$ of $n \times n$ matrices, asymptotic expansions of \eqref{MellinL} have been considered in \cite[Theorem 5.1]{Borgoetal:2017}. 
From a point of view of mathematical statistics, the number of variables $p(n)$ and the number
of measurements or observations $n$ are typically different. One aim of our paper is to study arbitrary Wishart matrices. We will consider $n-p(n)$ converging to
zero, or converging to a constant $c>0$, or $n-p(n)$ is growing at a certain rate with $n$. Moreover we will analyse the case of a fixed number of variables $p$. This
case will behave differently, see Theorem \ref{cumulant} and Theorem \ref{theoremmodL}.
A good overview of results for $\beta$-Laguerre ensembles is \cite{Bai:book} and \cite{Forrester:book}. 
In \cite{Jonsson:1982} one can find a very early result: the author proved a central limit theorem for $\det W_{n,n}^{L, 1}$, which is
$$
\frac{\log \det W_{n,n}^{L, 1} + n + \frac 12 \log n}{\sqrt{2 \log n}} \to N(0,1),
$$
where $N(0,1)$ denotes the standard Gaussian distribution. We will add the second order analysis.

\subsection{Jacobi ensembles / correlation coefficients}
Let $A_1$ and $A_2$ be $n_1 \times p(n)$ and $n_2 \times p(n)$ Gaussian matrices over $\re$, $\cn$ or $\qn$ with $p(n) \leq \min(n_1,n_2)$. The distribution of the matrix
$$
A_1^{\dagger} A_1 \bigl( A_1^{\dagger} A_1 + A_2^{\dagger} A_2 \bigr)^{-1}
$$
is called Jacobi ensemble. The model can be generalised to all $\beta >0$ as in the previous matrix models, see \cite{DumitriuEdelman:2002} and \cite{KillipNenciu:2004} for the corresponding tridiagonal models.
The joint density of the eigenvalues on the set $(0,1)^{p(n)}$ is given by
$$
\frac{1}{Z_{p(n),n_1, n_2, \beta}^J} \prod_{1 \leq j < k \leq p(n)} |\lambda_j - \lambda_k|^{\beta} \prod_{k=1}^{p(n)} \lambda_k^{\frac{\beta}{2} (n_1 - p(n) +1) - 1} 
\, (1 - \lambda_k)^{\frac{\beta}{2} (n_2 - p(n) +1) - 1}.
$$
One use of this joint density relates to {\it correlation coefficients} in multivariate statistics, see \cite[Section 3.6.1]{Forrester:book} for details.

Using Selberg integration from \cite[(17.1.3)]{Mehta:book}, we obtain
$$
Z_{p(n),n_1,n_2,\beta}^J = \prod_{k=1}^{p(n)} \frac{\Gamma(1+\frac{\beta}{2} k) \Gamma( \frac{\beta}{2}(n_2-p(n)) + \frac{\beta}{2} k) \Gamma(\frac{\beta}{2}(n_1-p(n)) + \frac{\beta}{2} k)}{\Gamma(1 + \frac{\beta}{2}) \Gamma (\frac{\beta}{2}(n_1 + n_2) + \frac{\beta}{2} (k-p(n)) )}.
$$
Using this Selberg formula, one obtains for the corresponding Mellin transform
$$
\E \biggl[ \biggl( \det W_{p(n),n_1,n_2}^{J, \beta} \biggr)^z \biggr] = \prod_{k=1}^{p(n)} \frac{\Gamma\bigl(
\frac{\beta}{2} (n_1-p(n)+k) +z \bigr) \Gamma \bigl( \frac{\beta}{2} (n_1+n_2-p(n)+k) \bigr) }{\Gamma\bigl(\frac{\beta}{2} (n_1 -p(n) +k)\bigr) \Gamma\bigl(\frac{\beta}{2} (n_1 +n_2 -p(n) +k) +z\bigr)} ,
$$
where $W_{p(n),n_1,n_2}^{J, \beta}$ denotes the $\beta$-Jacobi distributed random matrix of dimension  $p(n) \times p(n)$.
Hence with \eqref{L} we have

\begin{equation} \label{MellinJ}
\log \E \biggl[ \biggl( \det W_{p(n),n_1,n_2}^{J, \beta} \biggr)^z \biggr]  =L(p(n), n_1-p(n), \beta/2;z) - L(p(n), n_1+n_2-p(n), \beta/2;z).
\end{equation}
Asymptotic expansions in the case $p(n)=n_1 = \lfloor n \tau_1 \rfloor$ and $n_2 = \lfloor n \tau_1\rfloor$, for some $\tau_1, \tau_2 > 0$, were considered in \cite[Theorem 4.5, 4.7, 4.9, 4.11 and 4.13]{Borgoetal:2017}. 

The aim of our paper is to study the asymptotic behaviour
under less restrictive assumptions. The main results are given in Theorem \ref{theoremmodJ} and \ref{modJ2}.

\subsection{Ginibre ensembles}
We now consider an arbitrary $n \times n$ matrix $A$ whose entries are independent real or complex Gaussian random variables with mean zero and variance one.
Using the Selberg identity of the Laguerre ensemble, one obtains

\begin{equation} \label{MellinG}
\log \E \biggl[ \biggl( \det W_{n}^{G, \beta} \biggr)^z \biggr] = \frac{nz}{2} \log \bigl( \frac{2}{\beta} \bigr) + \log \prod_{k=1}^{n} 
\frac{\Gamma \bigl( \frac{\beta}{2} k + \frac{z}{2} \bigr)}{\Gamma \bigl( \frac{\beta}{2} k \bigr)} = \frac{nz}{2} \log \bigl( \frac{2}{\beta} \bigr)+ L(n,0,\beta/2;z),
\end{equation}
see for example \cite[(33), Section 3]{DoeringEichelsbacher:2012b}. For a central limit theorem with a Berry-Esseen bound see \cite{NguyenVu:2011}. For certain
refinements in the Gaussian case see \cite[Section 3]{DoeringEichelsbacher:2012b}. Second order asymptotics are included in \cite{Borgoetal:2017}, see Section 6.3.

\subsection{The threefold way due to Dyson and fixed-trace ensembles}
The most classical ensembles are the Hermite ensembles. Let $A$ be an $n \times n$ Gaussian matrix over $\re, \cn$ or $\qn$. The distribution of $\frac{A + A^{\dagger}}{2}$ is called the Gaussion orthogonal (GOE), unitary (GUE), and symplectic ensemble (GSE) respectively. The joint density function of the eigenvalues is given by
$$
\frac{1}{Z_n^H(\beta)} \prod_{1 \leq j < k \leq n} |\lambda_j - \lambda_k|^{\beta} \prod_{k=1}^n \exp \big( - \frac{\lambda_k^2}{2} \bigr),
$$
for $\beta=1,2,4$. Using Selberg integration from \cite[(17.6.7)]{Mehta:book}, we obtain
$$
Z_n^H(\beta) = (2 \pi)^{n/2} \prod_{k=1}^n \frac{\Gamma \bigl( 1 + k \frac{\beta}{2} \bigr) }{\Gamma \bigl(1 + \frac{\beta}{2} \bigr)}.
$$
Using a tridiagonal reduction algorithm, in \cite{DumitriuEdelman:2002}, a matrix model for other choices of $\beta >0$ was proved. But the lack of Selberg-integrals
makes it much harder to investigate these models. In \cite{Borgoetal:2017}, the study was restricted to the GUE model. We shall denote by $W_n^H$ a GUE random matrix. Here one can use the following
formula of the Mellin transform of the absolute value of the determinant, computed in \cite{MehtaNormand:1998}:
\begin{equation}  \label{MellinGUE}
\E \biggl[ \big| \det W_n^{H} \big|^z \biggr] = 2^{nz/2} \prod_{k=1}^{n} \frac{\Gamma\bigl(
\frac{z+1}{2} + \lfloor \frac k2 \rfloor \bigr)}{ \Gamma \bigl( \frac{1}{2} + \lfloor \frac k2 \rfloor\bigr) },
\end{equation}
which is well defined for any $z \in \cn$ with ${\operatorname{Re}}(z) >-1$.
The three ensembles GOE, GUE, GSE are also called {\it Wigner-Dyson ensembles A, AI and AII}. The background is the following.
The classical period of random matrix theory began in the late 1950s and early 1960s, when Wigner and Dyson proposed to model the discrete part of the
spectrum of the Hamiltonian of a complicated quantum system by the
spectrum of a suitable random matrix ensemble. To be a good model,
this ensemble had to share certain symmetries with the quantum
system. In his famous paper \cite{Dyson:1962}, Dyson adopted a set of
symmetry assumptions, which was motivated by the framework of
classical quantum mechanics, and classified those spaces of
matrices which are compatible with the given symmetries. He ended
up with the threefold way of hermitian matrices with real,
complex, and quaternion entries, i.e., precisely with those spaces
on which the familiar Gaussian orthogonal, unitary, and symplectic
ensembles (GOE, GUE, GSE) of classical random matrix theory are
supported. Dyson's threefold way is established in geometrical terms, without reference to
probability measures on the matrix spaces in question. In structural terms, the space of hermitian, real symmetric and quaternion real matrices can be viewed
as tangent spaces to, or infinitesimal versions of Riemannian symmetric spaces of type A, AI and AII respectively. 
A good overview of results for the Dyson ensembles is \cite{Zeitounibook}.

\noindent
{\it Fixed-trace ensembles} of random matrices were first considered by Rosenzweig and Bronk, see \cite[Chapter 19]{Mehta:book}. Universal limits for the eigenvalue correlation functions in the bulk of the spectrum in trace-fixed matrix ensembles are considered in \cite{GoetzeGordin:2008, GoetzeGordinLewina:2007}. 
Let us consider the GUE model. The maximum of $\prod_{k=1}^n |\lambda_k|$, subject to $\sum_{k=1}^n \lambda_k^2 =1$, is $n^{-n/2}$. Hence we consider the rescaled determinant
$$
n^{n/2} | \det W_n^{H,ft} |,
$$
where $\det W_n^{H,ft}$ is the product of the eigenvalues of the fixed-trace GUE ensemble. Here one can use the following
formula of the Mellin transform of $n^{n/2} | \det W_n^{H,ft} |$, computed in \cite{LeCaerDelannay:2003}: consider \cite[(20)]{LeCaerDelannay:2003}
with $\beta =2$ at the value $z+1$, see also \cite[page 256/257]{DoeringEichelsbacher:2012b}:
\begin{equation} \label{MellinGUEft}
\E \biggl[ \big| n^{n/2}  \det W_n^{H,ft} |^z \biggr] = \prod_{k=1}^{n} \frac{\Gamma\bigl(
\frac{z+1}{2} + \lfloor \frac k2 \rfloor \bigr)}{ \Gamma \bigl( \frac{1}{2} + \lfloor \frac k2 \rfloor\bigr) } \,
\biggl( \prod_{k=1}^{n} \frac{\Gamma\bigl(
\frac{z}{2} + \frac{n}{2} +\frac{k-1}{n} \bigr)}{ \Gamma \bigl( \frac{n}{2} + \frac{k-1}{n} \bigr) } \biggr)^{-1}. 
\end{equation}
In \cite{Borgoetal:2017}, asymptotic expansions of \eqref{MellinGUE} have been considered. We will add a precise asymptotic description for \eqref{MellinGUEft}.
This includes a central limit theorem, presice moderate deviations and Berry-Esseen bounds. All these results are new and presented in Section 6.5.

\subsection{The tenfold way}
In the 1990s, in the field of condensed matter physics, it was observed that random matrix models for so-called mesoscopic normal-superconducting hybrid structures
must be taken from the infinitesimal versions of further symmetric spaces, different from the path of the threefold way of Dyson. In \cite{Heinzner/Huckleberry/Zirnbauer:2005},
a classification similar to Dysons's way was developed, based on less restrictive
assumptions, thus taking care of the needs of modern mesoscopic
physics. Their list is in one-to-one correspondence with the
infinite families of Riemannian symmetric spaces as classified by
Cartan. In \cite{ES}, the corresponding random matrix
theory was introduced, with a special emphasis on large deviation principles. The seven new ensembles are listed in \cite{Forrester:book} 
in Section 3.1 and Section 3.3. as well as in \cite{ES}. All these new ensembles can be described by a certain matrix configuration (see the table on page 104 in \cite{Forrester:book}).
For any new ensemble, the joint density of the positive eigenvalues is of the form
\begin{equation} \label{density7}
\frac{1}{Z_{n,p(n), \beta, \mu}} \prod_{1 \leq j < k \leq p(n)} |\lambda_j^2 - \lambda_k^2|^{\beta} \, \prod_{k=1}^{p(n)} \lambda_k^{\beta \mu} \exp \bigl( - \beta \lambda_k^2 /2 \bigr)
\end{equation}
with some $\mu \in \re$.

\subsubsection{Chiral ensembles}
The three Chiral ensembles are defined in \cite[Definition 3.1.2]{Forrester:book}. Here we have to choose $\beta \in \{1,2,4\}$ and $\mu_{\operatorname{chiral}} = n-p(n) +1 - \frac{1}{\beta}$.

\subsubsection{Bogoliubov- de Gennes ensembles}
The other four ensembles are collected in Subsection 3.3.2 of \cite{Forrester:book}. Here, for the first ensemble, we have to choose $n=p(n)$, $\beta=1$ and $\mu =1$.
For the second (only if $n$ is even) $p(n) = n/2$,  $\beta=2$ and $\mu =0$. The third is given by  $p(n) = n/2$,  $\beta=4$ and $\mu =1/4$ if $n$ is even, and by
$p(n) = (n-1)/2$,  $\beta=2$ and $\mu =5/4$ if $n$ is odd. Finally, the fourth ensemble is given by  $p(n) = n$,  $\beta=2$ and $\mu =1$.
\medskip

\noindent
We will discuss the precise asymptotic behaviour of the seven ensembles. Therefore we have to calculate $Z_{n,p(n), \beta, \mu}$. We start at \cite[(17.6.6)]{Mehta:book}
with $\gamma=\beta/2$:
\begin{eqnarray*}
\int_{-\infty}^{\infty} \cdots \int_{-\infty}^{\infty}  \prod_{1 \leq j < k \leq p(n)} |\lambda_j^2 - \lambda_k^2|^{\beta} \, \prod_{k=1}^{p(n)} |\lambda_k|^{ 2 v -1} 
\exp \bigl( - \beta \lambda_k^2 /2 \bigr) \, d \lambda_1 \cdots d \lambda_k \\ 
& & \hspace{-10cm} = \bigl( \frac{2}{\beta} \bigr)^{v p(n) + \frac{\beta}{2} p(n)(p(n)-1)} 2^{-p(n)} \prod_{k=1}^{p(n)} \frac{ \Gamma(1 +  \frac{\beta}{2} k) \, \Gamma( v + \frac{\beta}{2}(k-1))}{\Gamma(1 + \frac{\beta}{2} )}.
\end{eqnarray*}
Note that the integrand is an even function in $\lambda_k$, so with $2v-1 = \beta \mu$, we obtain that
$$
Z_{n,p(n), \beta, \mu} = \bigl( \frac{2}{\beta} \bigr)^{ \bigl( \frac{\beta}{2} \mu + \frac 12 \bigr) p(n) + \frac{\beta}{2} p(n)(p(n)-1)} 
\prod_{k=1}^{p(n)} \frac{ \Gamma(1 +  \frac{\beta}{2} k) \, \Gamma( \bigl( \frac{\beta}{2} \mu + \frac 12 \bigr) + \frac{\beta}{2}(k-1))}{\Gamma(1 + \frac{\beta}{2} )}.
$$
Let us denote by $W_{n,p(n)}^{\beta, \mu}$ a random matrix with joint eigenvalue distribution \eqref{density7}. Using this Selberg formula, we obtain the corresponding Mellin transform
\begin{equation} \label{Mellin7}
\E \biggl[ \biggl( \det W_{n,p(n)}^{\beta, \mu} \biggr)^z \biggr] = \prod_{k=1}^{p(n)} 
\frac{\Gamma \bigl( \frac{\beta}{2} \mu + \frac 12 + \frac{z+1}{2} + \frac{\beta}{2}(k-1) \bigr)}{\Gamma( \bigl( \frac{\beta}{2} \mu + \frac 12 \bigr) + \frac{\beta}{2}(k-1))}.
\end{equation}
Hence for the three chiral ensembles ($\beta \in \{1,2,4\}$) we have

\begin{equation} \label{Mellin7ch}
\E \biggl[ \biggl( \det W_{n,p(n)}^{\beta, \mu_{\operatorname{chiral}}} \biggr)^z \biggr] = \prod_{k=1}^{p(n)} 
\frac{\Gamma \bigl( \frac{\beta}{2} (n-p(n) +k) + \frac{z+1}{2} \bigr) }{\Gamma \bigl( \frac{\beta}{2} (n-p(n) +k) \bigr)} = L(p(n), n-p(n), \beta/2; \frac{z+1}{2}).
\end{equation}
For the Bogoliubov- de Gennes ensembles, since $\mu$ is constant, the behaviour is different. Let us consider the case $n=p(n)$, $\beta=1$ and $\mu =1$.
Here we have
\begin{equation} \label{Mellin7BdG}
\E \biggl[ \biggl( \det W_{n,n}^{1,1} \biggr)^z \biggr] = \prod_{k=1}^{n} 
\frac{\Gamma \bigl( \frac{1}{2} (k+1)  + \frac{z+1}{2} \bigr) }{\Gamma \bigl(\frac{1}{2} (k+1) \bigr) } = L(n,1,1/2; \frac{z+1}{2}).
\end{equation}
Summarising, the asymptotic behaviour of the log-determinant of a Laguerre ensemble will lead to the asymptotic behaviour
of the product of non-negative eigenvalues in the seven new ensembles. All our results for the sum of the logarithms of the positive eigenvalues are new and
will be presented in Section 6.4.

\subsection{Gram ensembles and random simplices}
If for $p(n) \leq n$, $X_1, \ldots, X_{p(n)+1}$ are independent random points in $\re^n$ which are distributed according to a multivariate Gaussian distribution
with density $f(|x|) = (2 \pi)^{-n/2} \exp( - \frac 12 |x|^2)$, $x \in \re^n$, we denote by $VP_{n,p(n)}$ the $p(n)$-dimensional volume of the {\it parallelotope} spanned by the points $X_1, \ldots, X_{p(n)}$. This is the determinant of the corresponding Gram matrix. It is known, see \cite{Mathai:1999}, that for all $m \geq 0$ the moments of order $2m$ of the volume fulfill
$$
\log \E \bigl( (VP_{n,p(n)})^{2m} \bigr) = m p(n) \log 2 + \log \prod_{k=1}^{p(n)} \frac{\Gamma\biggl(
\frac{1}{2} (n-p(n)+k) +m \biggr)}{\Gamma\biggl(\frac{1}{2} (n -p(n) +k)\biggr)}.
$$
The formula is a consequence of the so-called Blaschke-Petkantschin formula from integral geometry. With \eqref{L}, hence we will study the asymptotics of
\begin{equation} \label{MellinPa1}
\log \E \bigl( (VP_{n,p(n)})^{z} \bigr) =  \frac z2 p(n) \log 2 + L \bigl( p(n), n-p(n), 1/2; z/2 \bigr),
\end{equation}
which is exactly the same as studying the asymptotic behaviour of the log-determinant of a Laguerre ensemble in the case $\beta=1$ for $z/2$ instead of $z$, see \eqref{MellinL}. Interestingly enough, the application of the Blaschke-Petkantschin formula is an alternative proof of the moment identity \eqref{MellinL}, which
in random matrix theory is proved with the help of Selberg integrals. 

\begin{remark}
In the theory of random matrices it is quite natural to consider a matrix size $p(n)$ growing with $n$. % and it is of special interest in high-dimensional statistics. 
In models of random geometry, the number of points $p(n)$ in $\re^n$ might not depend on $n$ and it might be a challenge to let $p(n)$ and $n$  tend to infinity simultaneously. Our results for $L \bigl( p(n), n-p(n), 1/2; z/2 \bigr)$ will imply this type of phenomena in high dimensions for free.
\end{remark}

In \cite{Mathai:2001}, the author studied the moments of order $2m$ of $VP_{n,p(n)}$ if the random points are distributed according to three other distributions, which
are called the Beta model, the Beta prime model and the spherical model. 
In the Beta model with parameter $\nu >0$, the i.i.d. points in the ball of radius 1 are distributed with respect to the density
$$
f(|x|) = \frac{1}{\pi^{n/2}} \frac{\Gamma \bigl( \frac{n + \nu}{2} \bigr)}{\Gamma \bigl( \frac{\nu}{2} \bigr)} ( 1 - |x|^2)^{(\nu-2)/2}, \quad x \in \re^n, |x| <1.
$$
Here for $z=2m$ we have
\begin{eqnarray} \label{MellinPa2}
\log \E \bigl( (VP_{n,p(n)})^{z} \bigr) &=&  p(n) \log \Gamma \bigl( \frac{n + \nu}{2} \bigr)  - p(n) \log \Gamma \bigl( \frac{n + \nu}{2}  + \frac z2\bigr)  \nonumber \\ &+&
 L \bigl( p(n), n-p(n), 1/2; z/2 \bigr).
\end{eqnarray}
In the Beta prime model with parameter $\nu >0$ the density is given by
$$
f(|x|) = \frac{1}{\pi^{n/2}} \frac{\Gamma \bigl( \frac{n + \nu}{2} \bigr)}{\Gamma \bigl( \frac{\nu}{2} \bigr)} ( 1 + |x|^2)^{-(\nu+n)/2}, \quad x \in \re^n
$$
and for $z=2m$ we have

\begin{equation} \label{MellinPa3}
\log \E \bigl( (VP_{n,p(n)})^{z} \bigr) =  p(n) \log \Gamma \bigl( \frac{\nu}{2} - \frac z2 \bigr)  - p(n) \log \Gamma \bigl( \frac{\nu}{2} \bigr)
 + L\bigl( p(n), n-p(n), 1/2; z/2 \bigr).
\end{equation}
Finally, in the spherical model the points are uniformly distributed on the sphere of radius 1 centred at the origin of $\re^n$. We have

\begin{equation} \label{MellinPa4}
\log \E \bigl( (VP_{n,p(n)})^{z} \bigr) =  p(n) \log \Gamma \bigl( \frac{n}{2} \bigr)  - p(n) \log   \Gamma \bigl( \frac{n+z}{2} \bigr) 
+ L \bigl( p(n), n-p(n), 1/2; z/2 \bigr)
\end{equation}
for any $z=2m$ with $m \in \na$, which is the Beta model with $\nu=0$. These identities are exactly given in \cite[Section 2.2.2, page 189]{Rouault:2007}. The second order analysis of the log-volume was already studied in \cite{Borgoetal:2017}, here in the case $p(n)=n$.  The authors established mod-Gaussian convergence in Theorem 4.5, and extended results in Theorem 4.9, 4.11 and 4.13. Our results for $p(n) \not=n$ are collected in Section 7.
\medskip

If we denote by $VS_{n,p(n)}$  the $p(n)$-dimensional volume of the {\it simplex} with vertices $X_1, \ldots, X_{p(n)+1}$, the moment formulas are very similar.
The following formulas were proved using the affine Blaschke-Petkantschin formula, see \cite{Miles:1971} and \cite{Chu:1993}. In the Gaussian model one obtains
\begin{equation} \label{MellinSim1}
\log \E \bigl( (p(n)! \, VS_{n,p(n)})^{z} \bigr) =  \frac z2 \log (p(n)+1) + \log \E \bigl( (VP_{n,p(n)})^{z} \bigr),
\end{equation}
where $\log \E \bigl( (VP_{n,p(n)})^{z} \bigr)$ is defined in \eqref{MellinPa1}.
\medskip

\noindent
In the Beta model we have
\begin{equation} \label{MellinSim2}
\log \E \bigl( (p(n)! \, VS_{n,p(n)})^{z} \bigr) = \log f(n,p(n), \nu,z) + \log \E \bigl( (VP_{n,p(n)})^{z} \bigr),
\end{equation}
where $\log \E \bigl( (VP_{n,p(n)})^{z} \bigr)$ is defined in \eqref{MellinPa2} and where 
$$
f(n,p(n), \nu,z) = \frac{\Gamma \biggl( \frac{n + \nu}{2} \biggr)}{\Gamma \biggl( \frac{n + \nu}{2} + \frac z2 \biggr)} \, 
\frac{ \Gamma \biggl( \frac{p(n)(n + \nu -2) + (n + \nu)}{2} + (p(n)+1) \frac z2 \biggr)}
{\Gamma \biggl( \frac{p(n)(n + \nu -2) + (n + \nu)}{2} + p(n) \frac z2 \biggr)}.
$$
\medskip

\noindent
In the Beta prime model we have
\begin{equation} \label{MellinSim3}
\log \E \bigl( (p(n)! \, VS_{n,p(n)})^{z} \bigr) = \log g(n,p(n), \nu,z) + \log \E \bigl( (VP_{n,p(n)})^{z} \bigr),
\end{equation}
where $\log \E \bigl( (VP_{n,p(n)})^{z} \bigr)$ is defined in \eqref{MellinPa3} and where
$$
g(n,p(n), \nu,z) = \frac{\Gamma \biggl( \frac{\nu}{2} - \frac z2\biggr)}{\Gamma \biggl( \frac{\nu}{2} \biggr)} \,
\frac{ \Gamma \biggl( \frac{(p(n)+1) \nu}{2} - p(n) \frac z2 \biggr)}{\Gamma \biggl( \frac{(p(n)+1)\nu}{2} - (p(n)+1) \frac z2 \biggr)}.
$$
\medskip

\noindent
Finally, in the spherical model we obtain
\begin{equation} \label{MellinSim4}
\log \E \bigl( (p(n)! \, VS_{n,p(n)})^{z} \bigr) = \log h(n,p(n), \nu,z) + \log \E \bigl( (VP_{n,p(n)})^{z} \bigr),
\end{equation}
where $\log \E \bigl( (VP_{n,p(n)})^{z} \bigr)$ is defined in \eqref{MellinPa4} and where
$$
h(n,p(n), \nu,z) = \frac{\Gamma \biggl( \frac{n}{2} \biggr)}{\Gamma \biggl( \frac{n}{2} + \frac z2\biggr)} \, 
\frac{ \Gamma \biggl( \frac{p(n)(n-2)+n}{2} + (p(n)+1) \frac z2 \biggr)}{\Gamma \biggl(  \frac{p(n)(n-2)+n}{2} + p(n) \frac z2 \biggr)},
$$
which is the same as the Beta model with $\nu=0$.
\medskip

\noindent
Summarising, the asymptotic behaviour of the log volume of random simplices is given by an expansion of $L \bigl( p(n), n-p(n), \frac 12; \frac z2 \bigr)$ as well as the asymptotic analysis of additional summands of the type
\begin{equation} \label{addsum1}
\log \Gamma \bigl( m(n,\nu) + z \bigr) - \log \Gamma \bigl( m(n,\nu) \bigr)  
\end{equation} 
with certain functions $m(n,\nu)$. Like in \cite{Borgoetal:2017}, proof of Lemma 4.2, the additional summands asymptotically behave like a polynomial of degree 2 in $z$, see Proposition \ref{Binetapp}. Hence these terms only modify the mean of the log-volume as well as the limiting function in the mod-Gaussian convergence, adding
a term of the form $e^{ \operatorname{const.} z^2}$.

\section{Mod-$\phi$ convergence and precise deviations}

The notion of mod-$\phi$ convergence has been studied and developed in the articles \cite{JacodKowalskiN:2011} and \cite{DelbaenKN:2015} -- see the new
textbook \cite{modbook} for more references. The main idea was to look for a natural renormalisation
of the characteristic functions of random variables which do not converge in law, instead of renormalisation of the random variables themselves. We will use the following definition of mod-$\phi$ convergence, see \cite[Definition 1.1.1]{modbook} and \cite[Definition 1.1]{mod2}:

\begin{definition} \label{moddef}
Let $(X_n)_n$ be a sequence of real-valued random variables, and let us denote by $\varphi_n(z)= \E [e^{zX_n}]$ their moment generating functions (Laplace transforms), all of which we assume to exist in a subset $D \subset \cn$. We assume that $D$ contains 0. 
%in a strip
%$$
%{\mathcal S}_{(c,d)} = \{ z \in \cn, c < \operatorname{Re}(z) < d \}
%$$
%with $c < d$ extended real numbers. 
Let $\phi$  be a non-constant infinitely divisible distribution with moment generating function $\int_{\re}e^{zx} \phi(dx) = \exp (\eta(z))$ that is well defined on $D$ ($\eta$ is called the L\'evy exponent). 
%${\mathcal S}_{(c,d)}$, 
Let $\psi$ be an analytic function that does not vanish on the real part of $D$
%${\mathcal S}_{(c,d)} $, 
such that locally uniformly on $D$ %$z \in {\mathcal S}_{(c,d)} $,
\begin{equation} \label{defmodphi}
\exp \bigl( - t_n \eta(z) \bigr) \varphi_n(z) \to \psi(z),
\end{equation}
where $(t_n)_n$ is some sequence going to $\infty$. We then say that $(X_n)_n$ converges mod-$\phi$ over $D$
%${\mathcal S}_{(c,d)} $, 
with parameters $(t_n)_n$ and limiting function $\psi : D \to \cn$. 
In the following we denote by
\begin{equation} \label{DefPsiN}
\psi_n(z):= \exp \bigl( - t_n \eta(z) \bigr) \varphi_n(z).
\end{equation}
When $\phi$ is the standard Gaussian distribution, we speak on mod-Gaussian convergence. 
%We shall always assume that 0 belongs to the band of convergence, i.e. $c < 0 < d$.
\end{definition}

\begin{remark}
We will mostly consider two subsets $D$:  a strip
$$
{\mathcal S}_{(c,d)} = \{ z \in \cn, c < \operatorname{Re}(z) < d \}
$$
with $c < 0 < d$ extended real numbers or $D = i \, \re$. Mod-$\phi$ convergence on $D= i \, \re$ corresponds to $\lim_{n \to \infty} \psi_n(i \xi) = \psi( i \xi)$
uniformly for $\xi$ in compact subsets of $\re$. 
\end{remark}

\begin{remark}
Recall that mod-$\phi$ convergence on an open subset $D$ of $\cn$ containing $0$ can only occur when the characteristic function of $\phi$ is analytic around $0$. Among the class of stable distributions, only Gaussian laws satisfy this property. Mod-$\phi$ convergence on $D= i \, \re$ can however be considered for any stable distribution $\phi$.
\end{remark}

It is easy to see that mod-Gaussian convergence on ${\mathcal S}_{c,d}$ implies a central limit theorem for a proper renormalisation of $(X_n)_n$, if $(t_n)_n$ goes to infinity. We consider 
$$
Y_n := \frac{X_n - t_n \eta'(0)}{\sqrt{t_n \eta''(0)}}.
$$
Indeed, for all $\xi \in \re$, by a Taylor expansion of $\eta$ around 0, we obtain
$$
\E \bigl[ \exp \big(i \xi Y_n \bigr) \bigr] = e^{\frac{\xi^2}{2}} \psi_n \bigl( \frac{i \xi}{\sqrt{t_n \eta''(0)}} \bigr) (1+o(1)) =  e^{\frac{\xi^2}{2}} \psi(0) (1 + o(1)),
$$
thanks to the uniform convergence of $\psi_n$ to $\psi$. But in fact there is much more information encoded in mod-Gaussian convergence. For instance, one can show that the normality zone is of order $o(\sqrt{t_n})$, this is meaning that the limit
$$
\lim_{n \to \infty} \frac{ P(Y_n \geq x)}{P(N(0,1) \geq x)} =1
$$
holds true for any $x =o(\sqrt{t_n})$. At the edges of such zone this approximation breaks down and the residue $\psi$ describes how to correct the Gaussian approximation of the tails. This is made more precise in the next two theorems. 

\begin{theorem} \label{cons1}
(Extended central limit theorem, Theorem 4.3.1 and Proposition 4.4.1 in \cite{modbook})

Consider a sequence $(X_n)_n$ that converges mod-$\phi$ on a band ${\mathcal S}_{(c,d)}$ with limiting distribution $\psi$ and parameters $t_n$, where $\phi$ is a non-lattice infinitely divisible law that is absolutely continuous w.r.t. the Lebesgue measure. Let $x=o((\sqrt{t_n})^{1/6})$, then
$$
P \bigl( X_n \geq t_n \eta'(0) + \sqrt{t_n \eta''(0)} x \bigr) = P (N(0,1) \geq x) (1+o(1)).
$$
In the case of mod-Gaussian convergence  the normality zone is $o(\sqrt{t_n})$.
\end{theorem}

For the next result we need the definition and some properties of the Legendre-Fenchel transform, a classical object in large deviation theory. The Legendre-Fenchel transform of a function $\eta$ is defined by
$$
F(x) = \sup_{h \in \re} (hx - \eta(h)).
$$
If $\eta$ is the logarithm of the moment generating function of a random variable (called cumulant generating function), then $F$ is a non-negative function and the unique 
$h$ minimising $hx - \eta(h)$, if it exits, is then defined by the implicit equation $\eta'(h)=x$. Here $h$ depends on $x$! One obtains $F(x) = xh - \eta(h)$ and $F'(x)=h$ and $F''(x) = h'(x) = \frac{1}{\eta''(h)}$.

\begin{theorem} \label{cons2}
(Precise deviations, Theorem 4.2.1 in \cite{modbook})

Consider a sequence $(X_n)_n$ that converges mod-$\phi$ on a band ${\mathcal S}_{(c,d)}$ with limiting distribution $\psi$ and parameters $t_n$, where $\phi$ is a non-lattice infinitely divisible law that is absolutely continuous w.r.t. the Lebesgue measure. Let $x \in (\eta'(0), \eta'(d)) $, then
$$
P \bigl( X_n \geq t_n x) = \frac{\exp(-t_n F(x))}{h \sqrt{2 \pi t_n \eta''(h)}} \psi(h) (1+o(1))
$$
where $h$ is given implicitly by $\eta'(h) =x$.
\end{theorem}

By applying the result to $(-X_n)_n$ one gets similarly
$$
P \bigl( X_n \leq t_n x) = \frac{\exp(-t_n F(x))}{|h| \sqrt{2 \pi t_n \eta''(h)}} \psi(h) (1+o(1))
$$
for $x \in (\eta'(c), \eta'(0))$. 

In case of mod-Gauss convergence we obtain $\eta(x) = x^2/2$ and therefore $h=x$ and $F(x) = x^2/2$ and $(\eta'(0), \eta'(d))= (0,d)$.

In \cite{mod2} estimates for the speed of convergence towards a limiting stable law in the setting of mod-$\phi$ convergence are given.
The notion of a {\it zone of control} was introduced in \cite{mod2} in the context of mod-stable convergence. We consider a stable distribution
$\phi_{c, \alpha, \delta}$ with scale parameter $c>0$, stability parameter $\alpha \in (0,2]$ and skewness parameter $\delta \in [-1,1]$, see \cite[Chapter 3]{Sato:book}. 
Any stable law $\phi_{c, \alpha, \delta}$ has a density $p_{c, \alpha, \delta}$ with respect to the Lebesgue measure. The standard normal law is one example ($c = \frac{1}{\sqrt{2 \pi}}, \alpha =2$ and $\delta=0$). Another example is the standard Cauchy distribution with $c=1, \alpha =1$ and $\delta=0$. If $(X_n)_n$ converges mod-$\phi_{c, \alpha, \delta}$
on $i \, \re$ with parameters $t_n$, then $\frac{X_n}{(t_n)^{1/\alpha}}$ converges to $\phi_{c, \alpha, \delta}$, if $\alpha \not= 1$. In the case $\alpha=1$
\begin{equation} \label{stableC}
\frac{X_n}{t_n} - \frac{2 c \delta}{\pi} \log t_n
\end{equation}
converges to $\phi_{c, \alpha, \delta}$, see \cite[Proposition 1.3]{mod2}.

\begin{definition} \label{defzone}
Let $(X_n)_n$ be a sequence of real random variables, $\phi_{c, \alpha, \delta}$ a stable law, and $(t_n)_n$ a growing sequence.Consider the following assertions:
\begin{enumerate}
\item Fix $v \geq 0$, $w>0$ and $\gamma \in \re$. There exists a zone of convergence $[- D \, 		t_n^{\gamma},  D \, t_n^{\gamma}]$, $D >0$, such that for all $\xi \in \re$ in this zone, 
$$
|\psi_n(i \xi) -1| \leq K_1 \, |\xi|^v \exp (K_2 {|\xi|^w})
$$
for some positive constants $K_1$ and $K_2$ that are independent of $n$. Here, $\psi_n$ is given by 
$ \psi_n(z) = e^{-t_n \eta_{c, \alpha, \delta}} \, \E[e^{z X_n}], $
where $\eta_{c, \alpha, \delta}$ denotes the L\'evy exponent of $\phi_{c, \alpha, \delta}$ (cf. \eqref{DefPsiN} ).

\item
One has
$$
\alpha \leq w , \qquad - \frac{1}{\alpha} \leq \gamma \leq \frac{1}{w-\alpha}, \qquad 0 < D \leq \biggl( \frac{c^{\alpha}}{2 K_2} \biggr)^{\frac{1}{w-\alpha}}.
$$
\end{enumerate}
If (a) holds for some parameters $\gamma > - \frac{1}{\alpha}$ and $v,w,D,K_1, K_2$, then (b) can always be forced by increasing $w$, and then decreasing $D$ and $\gamma$ in the bound of condition (a). If conditions (a) and (b) are satisfied, we say that $(X_n)_n$ has a {\it  zone of control}
 $[- D \, t_n^{\gamma},  D \, t_n^{\gamma}]$ and index of control $(v,w)$.
\end{definition}

The terminology of a zone of control does not mention the reference law $\phi_{c, \alpha, \delta}$ although it depends on the law. The law is considered to be fix.

\begin{remark} 
If $(X_n)_n$ has a {\it  zone of control}
 $[- D \, t_n^{\gamma},  D \, t_n^{\gamma}]$ and index of control $(v,w)$ and if $ - \frac{1}{\alpha} < \gamma$, then for $\alpha \not=1$ the sequence $\frac{X_n}{(t_n)^{1/\alpha}}$ and for $\alpha=1$ the sequence in \eqref{stableC}  converges to the law $\phi_{c, \alpha, \delta}$, see \cite[Proposition 2.3]{mod2}.
 In the definition of zone of control, one does not assume the mod-$\phi_{c, \alpha, \delta}$  convergence of $(X_n)_n$. However in our examples we shall indeed have mod-$\phi$ convergence with the same parameters $t_n$. We then speak of mod-$\phi$ convergence with a zone of control $[- D \, t_n^{\gamma},  D \, t_n^{\gamma}]$ and index of control $(v,w)$.
\end{remark}

The following result can be found in \cite[Theorem 2.16]{mod2}:

\begin{theorem} \label{BE}
(Rate of convergence)
Fix a reference stable distribution $\phi_{c, \alpha, \delta}$ and consider a sequence $(X_n)_n$ of random variables with a zone of control 
$[- D \, t_n^{\gamma},  D \, t_n^{\gamma}]$ and index of control $(v,w)$. Assume in addition that  $\gamma \leq \frac{v-1}{\alpha}$. 
If $Y$ denotes a random variable with law $\phi_{c, \alpha, \delta}$, then there exists a constant $C(D,v,K_1,\alpha,c)$ such that 
$$
d_{\operatorname{Kol}} \bigl( Y_n, Y \bigr) \leq  C(D,v,K_1,\alpha,c)  \frac{1}{t_n^{\gamma + \frac{1}{\alpha}}},
$$
where $d_{\operatorname{Kol}} \bigl( \cdot, \cdot \bigr)$ denotes the Kolmogorov distance and $Y_n$ is  $\frac{X_n}{t_n^{1/\alpha}}$ if $\alpha \not= 1$ and
the random variable in \eqref{stableC} if $\alpha=1$.
\end{theorem}

For an explicit form of the constant $C(D,v,K_1,\alpha,c)$ see \cite[Theorem 2.16]{mod2}. In the Gaussian case we have
\begin{equation} \label{constBE}
C(D,v,K_1,2,1/\sqrt{2}) = 3/(2 \pi) (2^{v-1} \Gamma(v/2) + 7/D \sqrt{\pi/2}).
\end{equation}
\medskip

\noindent
In \cite{DalBorgo2} the authors proved the following local limit theorem:

\begin{theorem} \label{localT}
(Local limit theorem)
Fix a reference stable distribution $\phi_{c, \alpha, \delta}$ and consider a sequence $(X_n)_n$ of random variables with a zone of control 
$[- D \, t_n^{\gamma},  D \, t_n^{\gamma}]$ and index of control $(v,w)$.
%Fix as the reference measure the standard Gaussian law and consider a sequence $(X_n)_n$ of random variables with a zone of control 
%$[- D \, t_n^{\gamma},  D \, t_n^{\gamma}]$ and index of control $(v,w)$.
Let $Y_n$ be  $\frac{X_n}{t_n^{1/\alpha}}$ if $\alpha \not= 1$ and
the random variable in \eqref{stableC} if $\alpha=1$.
Let $x \in \re$ and $B$ be a fixed Jordan measurable subset with strictly positive Lebesgue-measure $m(B)$. 
Then for every exponent $\mu \in (0, \gamma + \frac{1}{\alpha})$,
$$
\lim_{n \to \infty} (t_n)^{\mu} P \biggl( Y_n - x \in \frac{1}{t_n^{\mu}} B \biggr) = m(B) \, p_{c, \alpha, \delta}(x).
$$
\end{theorem}
\medskip

\section{Main Theorem}

All the representations of the moments of Gamma type motivate to consider the following key asymptotic expansion, which is a generalisation of Theorem 5.1
in \cite{Borgoetal:2017}. Denote by
$$
S_{\alpha} := \biggl\{ z \in \cn : -\alpha < \operatorname{Re}(z) \biggr\}.
$$

\begin{theorem} \label{MAIN}
For all $p \geq 1$ and any $z \in S_{\alpha}$ with $\alpha > 0$ and $|z| < \operatorname{const.} \alpha (p+l)^{1/6}$, we have
\begin{eqnarray} \label{main}
L(p,l,\alpha;z) & = & \log \biggl( \prod_{k=1}^p \frac{\Gamma(\alpha(k+l) + z)}{\Gamma( \alpha(k+l))} \biggr) \\
& = & \sum_{i=1}^3 T_i(p,l,\alpha;z) + T_4(l,\alpha;z) + T_5(l,\alpha;z) + R(p,l,\alpha;z), \nonumber
\end{eqnarray}
where $T_i(p,l,\alpha;z)$ are defined in \eqref{T1}, \eqref{T2}, \eqref{T3} for $i=1,2,3$, $T_4(l,\alpha;z)$ in \eqref{T4}, $T_5(l,\alpha;z)$ in \eqref{T5} and $R(p,l,\alpha;z)$ is defined in \eqref{R}.
\end{theorem}

\begin{proof}
The proof is a straightforward generalisation of the proof of Theorem 5.1 in \cite{Borgoetal:2017} applying
the Abel-Plana formula (see Theorem \ref{Abel-Plana} in the Appendix) which allows to evaluate even non-convergent sums, which cannot be handled applying Taylor expansion.

For a complex number  $z= |z| e^{i \arg(z)}$, $z \neq 0$, $\arg{z} \in (-\pi,\pi]$ we define the \textit{principal branch} of the logarithm as $\log(z) := \log(|z|) + i \arg(z)$.
As is customary in this framework, every equation $x=y$ involving complex logarithms is to be read as $\exp(x)=\exp(y)$.
This way, the classical identities $\log(xy) = \log(x) + \log(y)$ and $\log(x^y) = y \log(x)$ remain true for complex arguments $x$ and $y$.
For instance, we may write
\[ 0 = \log(e^{2 \pi i}) = 2 \pi i \log(e) = 2 \pi i.  \]

We have to consider the asymptotic of
$$
\sum_{k=1}^p \bigl( \log \Gamma(\alpha(k+l) + z) - \log \Gamma(\alpha(k+l)) \bigr)
$$
as $p=p(n)$ or $l=l(n)$ goes to infinity with $n \to \infty$. The complex Gamma function
for $z \in \cn$ with $\operatorname{Re}(z) >0$ is given by $\Gamma(z) = \int_0^{\infty} e^{-t} t^{z-1} \, dt$. The first Binet's formula for the logarithm of the Gamma function is given by
\begin{equation} \label{Binet1}
\log \Gamma(z) = \bigl( z - \frac 12 \bigr) \log z - z +1 + \int_0^{\infty} \varphi(s) (e^{-sz} - e^{-s}) \, ds, \quad \operatorname{Re}(z) >0.
\end{equation}
Here the function $\varphi$ is given by $\varphi(s) = \bigl( \frac 12 - \frac 1s + \frac{1}{e^s-1} ) \frac 1s$ and satisfies
for every $s \geq 0$ $0 < \varphi(s) \leq \lim_{s \to 0} \varphi(s) = \frac{1}{12}$.  Applying Binet's formula leads to
$$
\sum_{k=1}^p \bigl( \log \Gamma(\alpha(k+l) + z) - \log \Gamma(\alpha(k+l)) \bigr) = T_1(p,l,\alpha;z) + S_1(p,l,\alpha,z),
$$
where
$$
S_1(p,l,\alpha,z) := \sum_{k=1}^p \biggl( \bigl( \alpha(k+l) + z - \frac 12 \bigr) \log \bigl( \alpha(k+l) + z \bigr) -  \bigl(\alpha(k+l) - \frac 12 \bigr) \log \bigl(\alpha(k+l) \bigr) \biggr)
$$
and
\begin{eqnarray} \label{T1}
T_1(p,l,\alpha;z) & := & -pz + \sum_{k=1}^p \int_0^{\infty} \varphi(s) ( e^{-s (\alpha(k+l) + z)} - e^{-s \alpha(k+l)} )\, ds \nonumber \\
&& \hspace{-1cm} =   - pz - \int_0^{\infty} \frac{\varphi(s) ( e^{-s z} -1) e^{-s \alpha(p+l)}}{e^{\alpha s} -1} \, ds 
+   \int_0^{\infty} \frac{\varphi(s) ( e^{-s z} -1) e^{-s \alpha l}}{e^{\alpha s} -1} \, ds. 
\end{eqnarray}
Considering the term $S_1(p,l,\alpha,z)$ we obtain
\begin{eqnarray*}
S_1(p,l,\alpha,z) & = & \sum_{k=1}^p \bigl( \alpha(k+l) + z - \frac 12 \bigr) \log \biggl( 1 + \frac{z}{\alpha(k+l)} \biggr) 
 + z \sum_{k=1}^p \log \bigl(\alpha(k+l) \bigr) \nonumber \\
 & = &  \sum_{k=0}^{p-1} f_l(k) + T_2(p,l,\alpha;z) 
\end{eqnarray*}
where
\begin{equation} \label{T2}
T_2(p,l,\alpha;z) :=z \, \log \bigl( \alpha^p (1+l)(2+l) \cdots (p+l) \bigr)
\end{equation}
and where
$$
f_l(s) := \biggl( \alpha (s+1+l) + z - \frac 12 \biggr) \log \biggl( 1 + \frac{z}{\alpha (s+1+l)} \biggr).
$$
As in \cite{Borgoetal:2017}, we apply the Abel-Plana formula (Theorem \ref{Abel-Plana}) and obtain
$$
 \sum_{k=0}^{p-1} f_l(k) = T_3(p,l,\alpha;z) + S_2(p,l,\alpha,z)
 $$
with
\begin{eqnarray*}
T_3(p,l,\alpha;z) &=& \int_1^{1+p} \bigl( \alpha(s+l) + z - \frac 12 \bigr) \log \biggl( 1 + \frac{z}{\alpha(s+l)} \biggr) ds 
\\ &-& \frac 12 \bigl( \alpha (p +1+l) +z - \frac 12 \bigr) \log \biggl( 1 + \frac{z}{\alpha(p+1+l)} \biggr),
\end{eqnarray*}
and where $S_2(p,l,\alpha,z)$ is the remainder term in the Abel-Plana formula:
$$
S_2(p,l,\alpha,z) := T_4(l,\alpha;z)  + T_5(l,\alpha;z) - R(p,l,\alpha,z)
$$
with
\begin{equation} \label{T4}
T_4(l,\alpha;z) := \frac 12 \bigl( \alpha (1 +l)  + z - \frac 12 \bigr) \log \biggl( 1 + \frac{z}{\alpha(1+l)} \biggr)
\end{equation}
and
$$
T_5(l,\alpha;z) := i \int_0^{\infty} \frac{f_l(is) - f_l(-is)}{e^{2 \pi s} -1} \, ds
$$
and
\begin{equation} \label{R}
R(p,l,\alpha;z) := i \int_0^{\infty}   \frac{f_l(p+is) - f_l(p-is)}{e^{2 \pi s} -1} \, ds.
\end{equation}
Adapting the proof in \cite{Borgoetal:2017} we are able to show that 
 \begin{equation} \label{Rplz}
 R(p,l,\alpha;z) = {\mathcal O} \biggl( \frac{|z| + |z|^2}{p+l} \biggr),
 \end{equation}
 For some implied constant, depending only on $\alpha$.
 The term $T_3(p,l,\alpha;z)$ can be computed explicitly via integration by parts:
 \begin{eqnarray*}
 T_3(p,l,\alpha;z) &=& \int_{1+l}^{1+p+l} \bigl( \alpha s + z - \frac 12 \bigr) \log \biggl( 1 + \frac{z}{\alpha s} \biggr) ds 
\\ && \hspace{-2cm} -\frac 12 \bigl( \alpha (p +1+l) +z - \frac 12 \bigr) \log \biggl( 1 + \frac{z}{\alpha(p+1+l)} \biggr) \\
&& \hspace{-2cm} =\biggl( \alpha \frac{s^2}{2} + \bigl( z - \frac 12 \bigr) s \biggr) \log \biggl( 1 + \frac{z}{\alpha s} \biggr) \bigg|_{1+l}^{1+p+l} + \frac{z}{2} \int_{1+l}^{1+p+l} \frac{\alpha s + 2 z -1}{\alpha s +z} \, ds \\
\\ && \hspace{-2cm} -\frac 12 \bigl( \alpha (p +1+l) +z - \frac 12 \bigr) \log \biggl( 1 + \frac{z}{\alpha(p+1+l)} \biggr).
\end{eqnarray*}
Next we obtain that
\begin{eqnarray*}
 \int_{1+l}^{1+p+l} \frac{\alpha s + 2 z -1}{\alpha s +z} \, ds &=&  \int_{1+l}^{1+p+l} \biggl( \frac{\alpha s + z}{\alpha s +z}
  + \frac{z -1}{\alpha s +z} \biggr) \, ds \\
  &&\hspace{-2cm} = p + \frac{z-1}{\alpha} \log \biggl( s + \frac{z}{\alpha} \biggr) \bigg|_{1+l}^{1+p+l} \\
 && \hspace{-2cm} =p + \frac{z-1}{\alpha} \log (1+p+l) -  \frac{z-1}{\alpha} \log (1+l) \\
  && \hspace{-2cm} +  \frac{z-1}{\alpha} \log \biggl( 1 + \frac{z}{\alpha(1+p+l)} \biggr) -  \frac{z-1}{\alpha} \log \biggl( 1 + \frac{z}{\alpha(1+l)} \biggr). 
\end{eqnarray*}
Summarizing we obtain
\begin{eqnarray} \label{T3} 
T_3(p,l,\alpha;z) &=& \nonumber \\
&& \hspace{-2cm} \biggl( \alpha \frac{(1+p+l)(p+l)}{2} + (z- \frac 12)(p+l+ \frac 12) + \frac{z(z-1)}{2 \alpha} \biggr)  \log \biggl( 1 + \frac{z}{\alpha(1+p+l)} \biggr) \nonumber \\
&- & \biggl( \alpha \frac{(1+l)^2}{2} + (z - \frac 12)(1+l) + \frac{z(z-1)}{2 \alpha} \biggr)  \log \biggl( 1 + \frac{z}{\alpha(1+l)} \biggr) 
\nonumber \\
&+& \frac{pz}{2} +  \frac{z(z-1)}{2 \alpha} \log \biggl( 1 + \frac{p}{1+l} \biggr). 
\end{eqnarray}
Depending on whether $p$ or $l$ or both depend on $n$ we will apply an expansion on the logarithm at this point.
The last point is - in adaption of the proof of Theorem 5.1 in \cite{Borgoetal:2017} - to represent the term $T_5(l,\alpha;z)$
in terms of nice functions.
By definition of $f_l$ we obtain
\begin{eqnarray*}
f_l(is) - f_l(-is) &=& i \alpha (s+l) \log \biggl( 1 + \bigl( 1 + \frac{z}{\alpha} \bigr)^2 (s+l)^{-2} \biggr) \\
&&- i \alpha (s+l) \log \bigl( 1+(s+l)^{-2} \bigr) \\ &&+ 2i (\alpha +z - \frac 12 ) \biggl( \arctan \biggl( \frac{s+l}{1 + \frac{z}{\alpha}} \biggr) - \arctan (s+l) \biggr).
\end{eqnarray*}
Hence we obtain
\begin{eqnarray} \label{T5}
T_5(l,\alpha;z) & = & i \int_0^{\infty} \frac{f_l(is) - f_l(-is)}{e^{2 \pi s} -1} \, ds  \nonumber \\
& = & - \alpha \int_l^{\infty} \log \biggl( 1 + \bigl( 1 + \frac{z}{\alpha} \bigr)^2 s^{-2} \biggr)  \frac{s \, ds}{e^{2 \pi (s-l)} -1} \\
&+ & \alpha \int_l^{\infty} \log \bigl( 1+s^{-2} \bigr) \frac{s \, ds}{e^{2 \pi (s-l)} -1}  \nonumber \\ 
&-& (\alpha +z - \frac 12 ) \int_l^{\infty} \arctan \frac{s}{1 + \frac{z}{\alpha}} \frac{ds}{e^{2 \pi (s-l)} -1}  \nonumber \\
&+&  (\alpha +z - \frac 12 ) \int_l^{\infty} \arctan s \frac{ds}{e^{2 \pi (s-l)} -1}. \nonumber
\end{eqnarray}
This is the desired representation for our applications.
\end{proof}
\medskip

\section{The key asymptotics}

In all our classes of examples, $L(p(n), r(n), \beta/2;z)$ has to be considered, see \eqref{L}. Here, $(p(n))_n$ will be an increasing sequence of natural numbers,
whereas $(r(n))_n$ is a sequence of real numbers. We will assume that
\begin{equation} \label{condz}
|z| < \operatorname{const.} (\beta/2) \max (p(n), r(n))^{1/6} 
\end{equation}
and $z \in S_{\beta/2}$. 
As a consequence of Theorem \ref{MAIN}, with \eqref{T1} we obtain
\begin{eqnarray}  \label{T1rn}
T_1(p(n), r(n), \beta/2;z) & =& - p(n) z +   \int_0^{\infty} \frac{\varphi(s) ( e^{-s z} -1) e^{-s (\beta/2) r(n)}}{e^{s \beta/2} -1} \, ds \nonumber \\
& & + O \biggl( \frac{|z|}{p(n) + r(n)} \biggr).
\end{eqnarray}
This follows by applying the inequalities $e^x - 1 \geq x$ and $|e^z -1| \leq |z| e^{|z|}$, for any $x \geq 0$ and $z \in \cn$ respectively, to be able to bound
\begin{eqnarray*}
\bigg|  \int_0^{\infty} \frac{\varphi(s) (e^{-sz} -1)}{e^{s \beta/2} -1} e^{-s (\beta/2) (p(n) + r(n)) } \, ds \bigg| & \leq & \frac{1}{12}
\int_0^{\infty} \frac{|e^{-sz} -1|}{|e^{s \beta/2} -1|} e^{-s (\beta/2)(p(n) + r(n))} \, ds \nonumber \\
& \leq &  \frac{1}{12}
\int_0^{\infty} \frac{ s |z| |e^{s|z|}}{s \beta/2} e^{-s (\beta/2) (p(n) + r(n))} \, ds  \nonumber \\
&\leq& \frac{1}{6 (\beta/2)^2} \frac{|z|}{p(n) +r(n)},
\end{eqnarray*}
as soon as $|z| \leq (\beta/4) (p(n) + r(n))$, which is compatible with the assumption \eqref{condz}, see Theorem \ref{MAIN}. Precisely this estimate is presented in \cite{Borgoetal:2017}. 
For \eqref{T2} we apply $\log n! = n \log n -n + \frac 12 \log (2 \pi n) + O(1/n)$ and get
\begin{eqnarray} \label{T2rn}
T_2(p(n), r(n),\beta/2; z) & = & z \biggl( p(n) \log \frac{\beta}{2} + (p(n) + r(n)) \log (p(n) + r(n)) - p(n)  \nonumber \\
& & \hspace{-1cm} - r(n) \log r(n) + \frac 12 \log \biggl( 2 \pi \bigl( 1 + \frac{p(n)}{r(n)} \bigr) \biggr) \biggr) + O \biggl( \frac{|z|}{p(n) + r(n)} \biggr).
\end{eqnarray}
Expanding the logarithm in the first summand of $T_3(p(n),r(n),\beta/2;z)$ in \eqref{T3}, we obtain 
\begin{eqnarray} \label{T3rn}
T_3(p(n), r(n),\beta/2;z)  &=& z \frac{r(n)}{2} + z p(n) + \frac{2 z^2}{\beta} - \frac{z}{\beta}  + O \biggl( \frac{|z| + |z|^2 + |z|^3}{p(n)+r(n)} \biggr) \nonumber \\
&& \hspace{-4cm} -\biggl( \frac{\beta (1+r(n))^2}{4} + (z - \frac 12)(1+r(n)) + \frac{z(z-1)}{\beta} \biggr)  \log \biggl( 1 + \frac{z}{\frac{\beta}{2}(1+r(n))} \biggr)  \nonumber \\
&& +\frac{z(z-1)}{\beta} \log \biggl( \frac{1+p(n) + r(n)}{1+r(n)} \biggr). 
\end{eqnarray}
Next, \eqref{T4} reads
\begin{equation} \label{T4rn}
T_4(r(n), \beta/2;z) = \frac 12 \bigl( \beta/2 (1 +r(n)) +z - \frac 12 \bigr) \log \biggl( 1 + \frac{2z}{\beta(1 + r(n))}\biggr).
\end{equation}
Finally, $T_5(r(n), \beta/2;z)$ is given by \eqref{T5} and $R(p(n),r(n), \beta/2;z) = O \biggl( \frac{|z| + |z|^2}{p(n) + r(n)} \biggr)$.

\noindent
For the most part, we will apply the expansions for $r(n) \to \infty$ as $n \to \infty$. In this case, we expand the logarithm and observe after a small calculation
\begin{eqnarray} \label{T34rn}
T_3(p(n), r(n),\beta/2;z) & + & T_4(p(n),r(n), \beta/2;z) \nonumber \\
&& \hspace{-3cm} =  z p(n) + \frac{z(z-1)}{\beta} \log \biggl( \frac{1+p(n) + r(n)}{1+r(n)} \biggr) + O \biggl( \frac{|z| + |z|^2 + |z|^3}{p(n)+r(n)} \biggr).
\end{eqnarray}
\bigskip

In the statement of the Theorems below, $G$ denotes the Barnes $G$-function, see the Appendix. Moreover, we define
\begin{eqnarray} \label{Phi}
\Phi_{\alpha}(z) &: =& \alpha \log G\bigl( \frac{z}{\alpha} +1 \bigr) - \bigl( z - \frac 12 \bigr) \log \Gamma \bigl( \frac{z}{\alpha} +1 \bigr) \nonumber \\
& & + \int_0^{\infty} \biggl( \frac{1}{2s} - \frac{1}{s^2} + \frac{1}{s(e^s-1)} \biggr) \frac{e^{-sz} -1}{e^{s \alpha} -1} \, ds + \frac 34
\frac{z^2}{\alpha} + \frac{z}{2}.
\end{eqnarray}

\begin{remark} In \cite{Borgoetal:2017}, $\Phi_{\beta/2}$ was introduced at the beginning of section 4. Checking the proof of \cite[Theorem 5.1]{Borgoetal:2017},
we are sure that the penultimate summand has to be $\frac 32 \frac{z^2}{\beta}$.
\end{remark}

\begin{remark}
In Lemma 7.1 of \cite{Borgoetal:2017}, the authors proved that $\Phi_{\alpha}$ can be written as a finite sum of log-Gamma and log-Barnes $G$-functions. 
These expressions are simpler, since they do not depend on the integral 
$$
\int_0^{\infty} \biggl( \frac{1}{2s} - \frac{1}{s^2} + \frac{1}{s(e^s-1)} \biggr) \frac{e^{-sz} -1}{e^{s \alpha} -1} \, ds, 
$$
which does not have a closed formula for all $\alpha >0$. We only mention two examples. If $\beta=2$, one has for all $n \geq 1$ and any $z \in {\mathcal S}_{\alpha}$ with $|z| < \frac{\alpha}{4} n^{1/6}$ that
$$
\Phi_{1}(z) = \Phi_{\beta/2}(z) = \frac{z}{2} \log (2 \pi) - \log G(1+z).
$$
If $\beta=1$, for the same $z$ it holds that
$$
\Phi_{1/2}(z) = z \biggl( - \frac 12 \log \frac 12 + \frac 12 \log (2 \pi) \biggr) - \frac 12 \log G(1 + 2 z) - \frac 12 \biggl( \log \Gamma(\frac 12) - \log \Gamma (\frac 12 +z) \biggr).
$$
\end{remark}
\bigskip

We consider formula \eqref{MellinL} for the moments of the determinant of a $\beta$-Laguerre ensemble
and will obtain the following results, depending on the growth of the sequence $(n-p(n))_n$. Interestingly enough, in most of the cases, we will observe mod-Gaussian convergence. In some cases no mod-$phi$ or a non-Gaussian mod-$\phi$ convergence occurs.

\begin{theorem} \label{cumulant}
$L(p(n), n-p(n), \beta/2;z)$, defined in \eqref{L}, satisfies the following asymptotic expansion
locally uniformly on ${\mathcal S}_{\beta/2}$:
\begin{enumerate}
\item
{\bf Case $p(n)=n$:}
$$
L(n,0,\beta/2;z)= z \mu_{1}(n,n) + \frac{z^2}{\beta} \log n + \Phi_{\beta/2}(z) + o(1)
$$
with $\Phi_{\beta/2}(z)$ given by \eqref{Phi} and $\mu_1(n,n)$ defined in \eqref{mupn}.
\item
{\bf Case $n-p(n) \to 0$ as $n \to \infty$:}
$$
L(p(n),n-p(n), \beta/2;z) = z \mu_{1}(p(n),n) + \frac{z^2}{\beta} \log n + \Phi_{\beta/2}^{n,p(n)}(z) + o(1),
$$
where
\begin{eqnarray} \label{mupn}
\mu_{1}(p(n),n) &:= &\biggl( \frac 12 - \frac{1}{\beta} \biggr) \log \biggl( \frac{n}{n-p(n)} \biggr)  + \frac{n}{2} - \frac{3\, p(n)}{2} \nonumber \\
& + &n \log n -(n-p(n)) \log (n-p(n)) + p(n) \log \biggl( \frac{\beta}{2}
\biggr) ,
\end{eqnarray}
and $\Phi_{\beta/2}^{n,p(n)}(z)$ is a function depending on $n$ and $p(n)$ such that 
$$
\lim_{n \to \infty} \Phi_{\beta/2}^{n,p(n)}(z) = \Phi_{\beta/2}(z)
$$
for all $z \in \cn$ we are considering, and $ \Phi_{\beta/2}(z)$ given by \eqref{Phi}.
\item
{\bf Case $n-p(n) = c$ with  $c \in \na$ fixed:}
$$
L(p(n),c,\beta/2;z) = z \mu_{2}(p(n),c) + \frac{z^2}{\beta} \log \biggl( \frac{p(n) +1+c}{1+c} \biggr)  + \Phi_{\beta/2}^c(z) + o(1)
$$
with 
\begin{eqnarray} \label{munew}
\mu_{2}(p(n),n) & = & \frac 12 \log (p(n)+c) - \frac{1}{\beta}  \log(p(n) +1+c) \nonumber \\
& + & (p(n)+c) \log (p(n)+c) +p(n) \log \bigl( \frac{\beta}{2} \bigr)  -p(n).
\end{eqnarray}
Here $\Phi_{\beta/2}^c(z)$ is defined in \eqref{Phid}.
\item
{\bf Case $n-p(n) \to \infty$ as $n \to \infty$:}
$$
L(p(n), n- p(n), \beta/2;z) = z \mu_{3}(p(n),n) + \frac{z^2}{\beta} \log  \biggl( \frac{n}{n-p(n)} \biggr) + o(1)
$$
with
\begin{eqnarray} \label{munew2}
\mu_{3}(p(n),n) & := &  \biggl( \frac 12 - \frac{1}{\beta} \biggr) \log \biggl( \frac{n}{n-p(n)} \biggr) + \frac 12 \log (2 \pi) \nonumber \\
&+&  n \log n - (n-p(n)) \log (n -p(n)) + p(n) \log \biggl( \frac{\beta}{2} \biggr) -p(n).
\end{eqnarray}
\item 
{\bf Case $p(n)=p$ for some $p \in \na$ fixed:}
$$
L(p, n-p, \beta/2;z) = z \mu_{4}(p,n) + z \biggl( p \log \bigl( \frac{\beta}{2} \bigr) -p  + \frac 12 \log (2 \pi) \biggr) + o(1)
$$
with
\begin{equation} \label{mu3L}
\mu_{4}(p,n) :=   \biggl( \frac 12 - \frac{1}{\beta} \biggr) \log n -  \biggl( \frac 12 - \frac{1}{\beta} \biggr) \log (n-p) +  n \log n - (n-p) \log (n -p). 
\end{equation}
\end{enumerate}
\end{theorem}

\begin{remark}
We notice that the corresponding $\mu's$ are the expectations of the log-determinants up to a constant.
\end{remark}

\begin{proof}
{\bf (a) Case $p(n)=n$:} We will apply Theorem \ref{MAIN} with $p=p(n) =n$, $l = l(n) =n-p(n) =0$ and
$\alpha = \beta/2$. Now we are exactly in the situation of \cite[Theorem 5.1]{Borgoetal:2017}.  It is not obvious
to obtain this result directly from the representation in Theorem \ref{MAIN}. Therefore we give the proof. 
From \eqref{T1rn} we obtain $T_1(n,0,\beta/2;z) = - n z +  \int_0^{\infty} \frac{\varphi(s) (e^{-sz} -1)}{e^{s \beta/2} -1} \, ds + O \bigl( \frac{|z|}{n} \bigr)$, 
as soon as $|z| \leq \beta/4 n$, which is compatible with our assumption. 
Moreover, from \eqref{T2rn} we obtain 
$$
T_2(n,0,\beta/2 ; z) = z \bigl( n \log \beta/2 + n \log n -n + \frac 12 \log (2 \pi n) \bigr) +O(|z| /n).
$$
From \eqref{T3rn} it follows that
\begin{eqnarray} \label{T3Null}  
T_3(n,0,\beta/2;z) &=& nz + \frac{z(z-1)}{\beta} \log n - \frac{z}{\beta} + \frac{2 z^2}{\beta} \nonumber \\
&- & \biggl( \beta/4 + z - \frac 12 + \frac{z(z-1)}{\beta} \biggr)  \log \biggl( 1 + \frac{2z}{\beta} \biggr) 
+  O \biggl( \frac{|z| + |z|^2 + |z|^3}{n} \biggr).
\end{eqnarray}
Moreover, \eqref{T4rn} leads to $T_4(0,\beta/2;z) = \frac 12 \biggl( \frac{\beta}{2} + z - \frac 12 \biggr) \log \biggl( 1 + \frac{2z}{\beta} \biggr)$.
A nice fact is that $T_5(0,\beta/2;z)$ can be represented in terms of the Barnes $G$ function and the Gamma function,
which was presented in \cite[page 20]{Borgoetal:2017}. Applying \eqref{Binet2} and \eqref{Barnes2} we obtain
\begin{eqnarray*}
T_5(0,\beta/2;z) &= &\frac{\beta}{2} \log G \biggl(1 + \frac{2z}{\beta} \biggr) - \biggl(z - \frac 12 \biggr) \log \Gamma \biggl(1 + \frac{2z}{\beta} \biggr) \\
&& + \log  \biggl(1 + \frac{2z}{\beta} \biggr) \biggl( \frac{z^2}{\beta} - \frac{z}{\beta} + \frac{z}{2} - \frac 14 \biggr) - \frac{z^2}{2 \beta} + \frac{z}{\beta} + \frac{z}{2} - \frac{z}{2} \log (2 \pi).
\end{eqnarray*}
Putting all terms together, we conclude as on page 20 in \cite{Borgoetal:2017}. Note that $\frac z2 \log (2 \pi)$ is $T_2$ cancelled by the last summand in $T_5$.
\bigskip

In all other cases $p(n) \not= n$, and we choose $l(n) = r(n) = n - p(n)$ and $\alpha= \beta/2$ and observe from \eqref{T1rn}
\begin{equation} \label{T1pn}
T_1(p(n), n-p(n),\beta/2;z) = -p(n) z + \int_0^{\infty} \frac{\varphi(s) ( e^{-s z} -1) e^{-s \beta/2 (n-p(n))}}{e^{\beta s/2} -1} \, ds + O \biggl( \frac{|z|}{n} \biggr),
\end{equation}
and from \eqref{T2rn}
\begin{eqnarray} \label{T2pn}
T_2(p(n), n-p(n),\beta/2;z) & = & z \biggl( p(n) \log (\beta/2) + n \log n - p(n) \nonumber \\ 
&& \hspace{-3cm}  - (n-p(n)) \log (n-p(n)) + \frac 12 \log(2 \pi) + \frac 12 \log \bigl( \frac{n}{n-p(n)} \bigr) \biggl) + O \biggl( \frac{|z|}{n} \biggr),
\end{eqnarray} 
and with \eqref{T3rn}
\begin{eqnarray} \label{T3pn} 
T_3(p(n),n-p(n),\beta/2;z) &=& z \frac{n}{2} + \frac{2 z^2}{\beta} - \frac{z}{\beta}  + O \biggl( \frac{|z| + |z|^2 + |z|^3}{n} \biggr) \nonumber \\
&& \hspace{-5.2cm} -\biggl( \frac{\beta (1+n-p(n))^2}{4} + (z - \frac 12)(1+n-p(n)) + \frac{z(z-1)}{\beta} \biggr)  \log \biggl( 1 + \frac{z}{\frac{\beta}{2}(1+n-p(n))} \biggr)  \nonumber \\
&+& \frac{p(n)z}{2} +  \frac{z(z-1)}{\beta} \log \biggl( \frac{1+n}{1+n-p(n)} \biggr). 
\end{eqnarray}
Moreover, we have
\begin{equation} \label{T4pn}
T_4(n-p(n),\beta/2;z) = \frac 12 (\frac{\beta}{2} (1 +n-p(n))  + z - \frac 12 ) \log \bigl( 1 + \frac{2z}{\beta(1+n-p(n))} \bigr).
\end{equation}
$T_5(n-p(n),\beta/2;z)$ is defined in \eqref{T5}, and $R(p(n),n-p(n),\beta/2;z) = O \biggl(\frac{|z|+|z|^2}{n} \biggr)$.
Now we are prepared to prove the other cases.

{\bf (b) Case $n-p(n) \to 0$}: Intuitively, we will expect the same asymptotic behaviour as in the case $n=p(n)$.
First, we collect in $T_1, \ldots, T_5$ the $n$-dependent prefactors of $z$ to obtain the size of the
expected value of the log-determinant. It is $-p(n)$ in $T_1$ and
$$p(n) \log(\beta/2) + n \log n - p(n) +\frac 12 \log(2\pi)  -(n-p(n)) \log(n-p(n)) + \frac 12 \log \bigl( \frac{n}{n-p(n)} \bigr)$$ 
in $T_2$. The $n$-dependent prefactor of $z$ in $T_3$ is $\frac{n}{2} + \frac{p(n)}{2} - \frac{1}{\beta} \log \bigl( \frac{n}{n-p(n)} \bigr)$, see
\eqref{T3pn}.  We obtain $\mu_1(p(n),n)$ in \eqref{mupn}.
The $n$ dependent prefactor of $z^2$ is $\frac{\log n}{\beta}$, see
\eqref{T3pn}. The sum of the remaining terms (without the $O$-terms) are defined
to be $\Phi_{\beta/2}^{n,p(n)}(z)$ which converges to $\Phi_{\beta/2}(z)$ as $n \to \infty$ (see case {\bf (a)}). 
This can be shown easily and the details are left to the reader. Notice that $\frac{z}{2} \log (2 \pi)$ is cancelled by the last summand of the limit of $T_5$.

{\bf (c) Case $n-p(n) =c$ for some fixed $c \in \na$}: Obviously the sums of all $n$-dependent prefactors of $z$ and $z^2$
in $T_1, \ldots, T_5$ are $\mu_2(p(n),c)$, and $\frac{1}{\beta} \log \bigl( \frac{p(n)+1+c}{1+c} \bigr)$ respectively. 
%We observe that now
%\begin{eqnarray*}
%\mu_{1}(p(n),n) &= &\biggl( \frac 12 - \frac{1}{\beta} \biggr) \log \biggl( \frac{n}{c} \biggr)  + \frac{n}{2} - \frac{3\, p(n)}{2} \nonumber \\
%& + &n \log n - c \log c + p(n) \log \biggl( \frac{\beta}{2} \biggr).
%\end{eqnarray*}
The terms which do not depend on $n$
are 
$$
U_1(c, \beta/2;z) =  \int_0^{\infty} \frac{\varphi(s) ( e^{-s z} -1) e^{-s \beta c/2}}{e^{(\beta/2) s} -1} \, ds
$$ 
from $T_1$, $U_2(c, \beta/2;z) = z \bigl( \frac{1}{2} (\log \bigl( \frac{2 \pi}{c} \bigr) - c \log (c) \bigr)$ from $T_2$, 
\begin{eqnarray*} 
U_3(c, \beta/2;z) & = & \frac{cz}{2} +\frac{2 z^2}{\beta} - \frac{z}{\beta}  +  \frac{z}{\beta} \log (1+c) \\
&& -\biggl( \frac{\beta (1+c)^2}{4} + (z - \frac 12)(1+c) + \frac{z(z-1)}{\beta} \biggr)  \log \biggl( 1 + \frac{z}{\frac{\beta}{2}(1+c)} \biggr)  
\end{eqnarray*}
from $T_3$,
$$
U_4(c, \beta/2;z) =  \frac 12 (\frac{\beta}{2} (1 +c)  + z - \frac 12 ) \log \bigl( 1 + \frac{z}{\frac{\beta}{2}(1+c)} \bigr)
$$
from $T_4$ and $U_5(c,\beta/2;z) = T_5(c, \beta/2;z)$. The result follows with 
\begin{equation} \label{Phid}
\Phi_{\beta/2}^c(z) := \sum_{j=1}^5 U_j(c,\beta/2;z).
\end{equation}

{\bf (d) Case $n-p(n) \to \infty$ as $n \to \infty$}: In this case we obtain
$$
T_1(p(n), n-p(n),\beta/2;z) = - p(n) z +  O \biggl( \frac{|z|}{n} + \frac{|z|}{n-p(n)}\biggr)
$$
and
\begin{eqnarray*} 
T_2(p(n), n-p(n),\beta/2;z) & = & z p(n) \log (\beta/2) + z n \log n -z p(n) + z \frac 12 \log(2 \pi n) \nonumber \\
& &  \hspace{-4cm} -z \bigl(n-p(n)\bigr) \log \bigl(n-p(n)\bigr)  - z \frac12 \log\bigl(2\pi(n-p(n))\bigr)  + O \biggl( \frac{|z|}{n} + \frac{|z|}{n-p(n)}\biggr).
\end{eqnarray*}
With the notions of the proof of Theorem \ref{MAIN}, we obtain that
$$
\sum_{k=0}^{p(n)-1} f_{n-p(n)}(k) = \sum_{k=0}^{n-1} f_0(k) - \sum_{k=0}^{n-p(n)-1} f_0(k).
$$
Since 
$$ \sum_{k=0}^{n-1} f_0(k) = T_3(n,0, \beta/2;z) + T_4(0, \beta/2;z) + T_5(0, \beta/2;z)- R(n,0,\beta/2;z)$$ and 
$$\sum_{k=0}^{n-p(n)-1} f_0(k) = T_3(n-p(n),0, \beta/2;z) + T_4(0, \beta/2;z) + T_5(0, \beta/2;z)- R(n-p(n),0,\beta/2;z),
$$
we obtain 
\begin{eqnarray} \label{T3Nullexp}
\sum_{k=0}^{p(n)-1} f_{n-p(n)}(k) & =& T_3(n,0, \beta/2;z) -T_3(n-p(n),0, \beta/2;z) \nonumber \\
&+& {\mathcal O} \biggl( \frac{|z| + |z|^2}{n} \biggl) +{\mathcal O} \biggl( \frac{|z| + |z|^2}{n-p(n)} \biggr).
\end{eqnarray}
With \eqref{T3Null} we get
$$
T_3(n,0, \beta/2;z) -T_3(n-p(n),0, \beta/2;z) = z p(n) + \frac{z(z-1)}{\beta} \log \biggl( \frac{n}{n-p(n)} \biggr),
$$
and the result follows.

{\bf (e) Case $p(n)=p$ for a fixed $p \in \na$}:
From the formulas in the proof of the previous case ($n-p(n) \to \infty$), we observe that
$T_1(p,n-p,\beta/2;z) = -pz + O \bigl( \frac{|z|}{n} \bigr)$ and
\begin{eqnarray*} 
T_2(p, n-p,\beta/2;z) & = & z p \log (\beta/2) + z n \log n -z p + z \frac 12 \log(2 \pi n) \nonumber \\
& &  \hspace{-2cm} -z \bigl(n-p\bigr) \log \bigl(n-p\bigr)  - z \frac12 \log\bigl(2\pi(n-p)\bigr)  + O \biggl( \frac{|z|}{n} \biggr).
\end{eqnarray*}
Moreover,
$$
T_3(n,0, \beta/2;z) -T_3(n-p,0, \beta/2;z) = z p + \frac{z(z-1)}{\beta} \log \biggl( \frac{n}{n-p} \biggr).
$$

Combining these terms as in the case before ($n-p(n) \to \infty$), we obtain that the expectation of the log-determinant is of size $\mu_3(p,n)$, and the result follows.
\end{proof}
\medskip

\section{Random matrix ensembles} \label{subLag}

\subsection{$\beta$-Laguerre ensembles}

A direct consequence of Theorem \ref{cumulant} is mod-$\phi$ convergence for the shifted log-determinants of the
considered $\beta$-Laguerre ensemble. We observe that mod-$\phi$ convergence sometimes fails.
Recall that by \eqref{MellinL}, we have
$$
\log \E \biggl[ \exp \bigl( z \, \log \bigl( \det W_{n,n}^{L, \beta} \bigr) \bigr) \bigg] = z p(n) \log 2 + L(p(n),n-p(n),\beta/2;z).
$$

\begin{theorem} \label{theoremmodL}
Mod-$\phi$ convergence for the log-determinant of $\beta$-Laguerre ensembles
\begin{enumerate}
\item
{\bf Case $p(n)=n$:}
The sequence $\bigl( X_1^L(n) := \log \bigl( \det W_{n,n}^{L, \beta} \bigr) - \mu_1(n,n) - n \log 2\bigr)_n$ converges {\bf mod-Gaussian} on ${\mathcal S}_{\beta/2}$
with parameter $t_n = \frac{2}{\beta} \log n$ and limiting function $\Psi(z) = \exp ( \Phi_{\beta/2}(z) )$. Here $\mu_1(n,n)$ is defined in \eqref{mupn}.
\item
{\bf Case $n-p(n) \to 0$ as $n \to \infty$:}
The sequence $\bigl( X_2^L(n) := \log \bigl( \det W_{n,p(n)}^{L, \beta} \bigr) - \mu_1(p(n),n) - p(n) \log 2 \bigr)_n$ converges {\bf mod-Gaussian} on ${\mathcal S}_{\beta/2}$
with parameter $t_n = \frac{2}{\beta} \log n$ and limiting function $\Psi(z) = \exp ( \Phi_{\beta/2}(z) )$. Here $\mu_1(p(n),n)$ is defined in \eqref{mupn}.
\item
{\bf Case $n-p(n) = c$ with  $c \in \na$ fixed:}
The sequence $\bigl( X_3^L(n) := \log \bigl( \det W_{n,p(n)}^{L, \beta} \bigr) - \mu_2(p(n),c) - p(n) \log 2\bigr)_n$ converges {\bf mod-Gaussian} on ${\mathcal S}_{\beta/2}$
with parameter $t_n = \frac{2}{\beta} \log \bigl( \frac{p(n) +1+c}{1+c} \bigr)$ and limiting function $\Psi(z) = \exp ( \Phi_{\beta/2}^c(z) )$. Here $\mu_2(p(n),c)$ is defined in \eqref{munew}.
\item
{\bf Case $n-p(n) \to \infty$ as $n \to \infty$:}
The sequence $\bigl( X_4^L(n) := \log \bigl( \det W_{n,p(n)}^{L, \beta} \bigr) - \mu_3(p(n),n)- p(n) \log 2 \bigr)_n$ converges {\bf mod-Gaussian} on ${\mathcal S}_{\beta/2}$
with parameter $t_n = \frac{2}{\beta} \log \biggl(\frac{n}{n-p(n)} \biggr)$ and limiting function $1$, whenever $p(n)$ is chosen such that $n-p(n) = o(n)$. Hence for any sequence $p(n)$ with $\frac{p(n)}{n} \to c \in [0,1)$, no mod-$\phi$ convergence takes place.  Here $\mu_3(p(n),n)$ is defined in \eqref{munew2}
\item
{\bf Case $p(n)=p$ for a fixed $p \in \na$:} The sequence $\bigl( \log \bigl( \det W_{n,p}^{L, \beta} \bigr) - \mu_4(p,n) - p \log 2\bigr)_n$ does not converge in the sense of mod-$\phi$
convergence. But the sequence $\bigl( n \log \bigl( \det W_{n,p}^{L, \beta} \bigr) - n (p \log n -p) \bigr)_n$ converges {\bf mod-$\phi$} on $i \, \re$
with parameter $t_n = p \, n$ and limiting function 
$$
\psi(z) = \bigl( 1 + \frac{2z}{\beta} \bigr)^{-\frac{\beta(p-1)p}{4} - \frac p2}.
$$
Here $\phi$ is such that the L\'evy exponent is
\begin{equation} \label{stable1}
\eta(z) = \log \int_{\re} e^{zx} \phi(dx) = - \frac{\beta}{2} \log \beta + \bigl( z + \frac{\beta}{2} \bigr) \log 2 +
\bigl( z + \frac{\beta}{2} \bigr) \log   \bigl( z + \frac{\beta}{2} \bigr).
\end{equation}
%and
%$\mu_4(p) =  p \, \log n -p$.
\end{enumerate}
\end{theorem}

Summarising, we obtain mod-$\phi$ convergence for the centred version of $\log (\det W_{n,p(n)}^{L,\beta})$:
\medskip

\hspace{-1cm}{
\begin{tabular}{|c|c|c|c|c|} \hline
condition & centred version of & mod-$\phi$ & $t_n$ & limiting function \\ \hline
$p(n) =n$  & $\log (\det W_{n,n}^{L,\beta})$ &mod-$N(0,1)$ & $\frac{2}{\beta} \log n$ & $\exp (\Phi_{\beta/2}(z))$ \\ \hline
$n-p(n) \to 0$  & $\log (\det W_{n,p(n)}^{L,\beta})$ &mod-$N(0,1)$ & $\frac{2}{\beta} \log n$ & $\exp (\Phi_{\beta/2}(z))$ \\ \hline
$n-p(n)=c$  & $\log (\det W_{n,p(n)}^{L,\beta})$ &mod-$N(0,1)$ & $\frac{2}{\beta} \log \bigl( \frac{p(n) +1+c}{1+c} \bigr)$ & $\exp (\Phi_{\beta/2}^c(z))$ \\ \hline
$n-p(n) =o(n)$ & $\log (\det W_{n,p(n)}^{L,\beta})$& mod-$N(0,1)$ & $\frac{2}{\beta} \log \bigl( \frac{n}{n-p(n)} \bigr)$ & $1$ \\ \hline
$p(n)=p $ & $ n\, \log (\det W_{n,p}^{L,\beta})$ & mod-$\phi$ on $i\, \re$& $p\, n$ & $ \bigl( 1 + \frac{2z}{\beta} \bigr)^{-\frac{\beta(p-1)p}{4} - \frac p2}$ \\ \hline
\end{tabular}}
%\medskip

\begin{proof}
The results, in most of the cases, follow directly form Theorem \ref{cumulant}. The only fact which has to be proven is case (e). We need to prove that
$\bigl( n \log \bigl( \det W_{n,p}^{L, \beta} \bigr) - n \mu_3(p) \bigr)_n$ converges mod-$\phi$. 
Note that for any ansatz considering $\bigl( n^{\alpha} \log \bigl( \det W_{n,p}^{L, \beta} \bigr) - n \mu_3(p) \bigr)_n$ for some $\alpha \leq 1$,
only the choice $\alpha=1$ leads to an appropriate asymptotic expansion. 
Interestingly enough, we will not apply Theorem \ref{MAIN}. The reason is that $n R(p, n-p,z) = O(|z| + |z|^2)$
will not be sufficient to achieve convergence. For the finite number of $p$ factors, we alternatively apply Stirling's formula, which reads
\begin{equation} \label{Stirling}
\Gamma(a z +b) = \sqrt{2 \pi} \exp( - az) (az)^{az+b - \frac{1}{2}} (1 +O(1/z))
\end{equation}
as $|z| \to \infty$, $a>0$, $b\in \re$ and $|\operatorname{arg}  z| < \pi$, see \cite[page 257]{AbramowitzStegun:1964}.
Applying Stirling's formula for $\Gamma\bigl(\frac{\beta}{2} (n-p+k) + n z \bigr)$ and $\Gamma\bigl(\frac{\beta}{2} (n -p +k)\bigr)$
leads to
\begin{align*} 
2^{p \, z\, n} \prod_{k=1}^{p} & \frac{\Gamma\bigl(
\frac{\beta}{2} (n-p+k) +n z  \bigr)}{\Gamma\bigl(\frac{\beta}{2} (n -p +k)\bigr)} \\
& = 2^{z\, p\, n} e^{-z \, p \,n} n^{z \, p \,n} \prod_{k=1}^p \frac{
\bigl( \beta/2 +z \bigr)^{n ( \beta/2 +z  ) + \beta/2(k-p) - 1/2}}{ \bigl( \beta/2)^{n \beta/2  + \beta/2(k-p) - 1/2}} +o(1)\\
& = 2^{z \, p \, n} e^{-z \, p \, n} n^{z\, p \, n}  \frac{\bigl( \beta/2 +z \bigr)^{pn \bigl( \frac{\beta}{2} + z\bigr) - \frac{\beta p(p-1)}{4} - \frac p2}}{\bigl( \beta/2)^{p n \frac{\beta}{2}  - \frac{\beta(p-1)p}{4}  - \frac p2}} +o(1).
\end{align*}
Hence
\begin{eqnarray} \label{StirAnw}
\log \biggl( \E \biggl[ \biggl( \det W_{n,p}^{L, \beta} \biggr)^{nz} \biggr]  \biggl) &=& nz \bigl( p \, \log n -p\bigr) \nonumber \\
& & \hspace{-1cm}+ p\, n \biggl( -\frac{\beta}{2} \log \beta  +  \bigl( z + \frac{\beta}{2} \bigr) \log 2 + \bigl(z + \frac{\beta}{2} \bigr) \log  \bigl(z + \frac{\beta}{2} \bigr) \biggr) \nonumber\\
&& \hspace{-1cm}+ \biggl( -\frac{\beta p(p-1)}{4} - \frac p2 \biggr) \log \bigl( 1 +\frac{2z}{\beta} \bigr) +o(1).
\end{eqnarray}
This is true for any $z \in \cn$ with  $|\operatorname{arg}  z| < \pi$, especially for all $z = i \xi$ with $\xi \in \re$.
Next we discuss $\phi$. We observe that $\phi$  is a non-constant infinitely divisible distribution. Moreover, one can find a tilted,
totally skewed 1-stable distribution such that the corresponding L\'evy exponent $\eta$ is \eqref{stable1}. This means that $\phi=\phi_{c,1,-1}$
for a certain $c$, depending on $\beta$ in the sense of the definition given in section 3. These distributions are known to be infinitely divisible.
For details, see \cite[Section 1.2]{stablebook}, particularly Proposition 1.2.12 as well as \cite[Chapter 2]{Sato:book}. Hence we have proved mod-stable
convergence on $i \, \re$.
\end{proof}

A consequence of Theorem \ref{theoremmodL}, case (e) is:

\begin{corollary}
Weak convergence of the log-determinant of $\beta$-Laguerre ensembles if $p(n)=p$ for a fixed $p \in \na$.
%In case (e) in Theorem \ref{theoremmodL} we 
Consider
$$
Y_n^L(\beta,p,c) := \frac 1p \log \bigl( \det W_{n,p}^{L, \beta} \bigr) + \frac{2c}{\pi} \log (pn).
$$
We conclude that  $Y_n^L(\beta,p,c)$ converges weakly to $\phi_{c,1,-1}$.
% \biggr) \leq C(D,1,K_1,1,c) \, \frac{1}{p \, n},%$$
\end{corollary}

A direct consequence of Theorems \ref{cons1} and \ref{cons2} are the following two results:

\begin{theorem} \label{exCLT}
Extended central limit theorems for log-determinants of $\beta$-Laguerre ensembles

In all cases (a)-(d)  (in case (d) only if $n-p(n) = o(n)$) in Theorem \ref{theoremmodL}, for $y = o( \sqrt{\log n})$, we observe
$$
P \biggl( X_i^L(n) \geq y \sqrt{\frac{ 2 \log n}{\beta}} \biggr) = P (N(0,1) \geq y) (1 +o(1))
$$ 
for $i=1,\ldots,4$. 
%In case (e), where $p(n)=p$ for a fixed $p \in \na$, we observe  for $y= o( \sqrt{n})$ 
%$$
%P \biggl(  n \log \bigl( \det W_{n,p}^{L, \beta} \bigr) - n \mu_4(p) \geq y \sqrt{p \, n} \biggr) = P (N(0,1) \geq y) (1 +o(1)).
%$$
\end{theorem} 

\begin{theorem} \label{pLDP}
Precise deviations for log-determimants of $\beta$-Laguerre ensembles

In case (a) and (b)  in Theorem \ref{theoremmodL}, for $x >0$, we obtain
$$
P \biggl( X_i^L(n) \geq \frac{x\, 2 \log n}{\beta} \biggr) = \frac{e^{- \frac{x^2}{2} \frac{2 \log n}{\beta}}}{x \sqrt{\frac{ 4 \pi \log n}{\beta}}}
\exp (\Phi_{\beta/2}(x)) (1 +o(1))
$$ 
for $i=1,2$. In case (c), we obtain for all $x>0$ that
 $$
P \biggl( X_3^L(n) \geq \frac{2 x \log n}{\beta} \biggr) = \frac{e^{- \frac{x^2 \log n}{\beta}}}{x \sqrt{\frac{ 4 \pi \log n}{\beta}}}
\exp (\Phi_{\beta/2}^c(x)) (1 +o(1)).
$$ 
In case (d), if $n-p(n) =o(n)$, we obtain for all $x>0$ that
$$
P \biggl( X_4^L(n) \geq \frac{2 x \log n}{\beta} \biggr) = \frac{e^{- \frac{x^2 \log n}{\beta}}}{x \sqrt{\frac{ 4 \pi \log n}{\beta}}} (1 +o(1)).
$$ 
%In case (e), where $p(n)=p$ for a fixed $p \in \na$, we obtain for $x>0$ 
%$$
%P \biggl( \log \bigl( \det W_{n,p}^{L, \beta} \bigr) - \mu_4^L(p) \geq x p \biggr) = \frac{\exp(- p \, n F(x))}{h \sqrt{2 \pi p\, n \eta''(h)}} \psi_{\beta/2}^p(h) (1+o(1)).
%$$ 
%Here
%\begin{equation} \label{rate}
%F(x) = \sup_{a \in \re}  \bigl( a x - \eta(a) \bigr) = \exp( x - \log 2 -1)  - \frac{\beta}{2} x + \frac{\beta}{2} \log \beta,
%\end{equation}
%$\eta$ given by \eqref{stable1}, $h = h_x = \exp(x - \log 2 -1) - \frac{\beta}{2}$ solves $\eta'(h)=x$, $\eta''(h) = \exp^{-1}(x - \log 2 -1)$ and 
%$$
%\psi_{\beta/2}^p(h) =  \bigl( 1 + \frac{2h}{\beta} \bigr)^{-\frac{\beta(p-1)p}{4} - \frac p2}.
%$$
\end{theorem}

For the next result, please recall the definition of a large deviation principle, see \cite[Section 1.2]{Dembo/Zeitouni:LargeDeviations}.

\begin{corollary} \label{LDP}
Large and moderate deviations principles for log-determinants of $\beta$-Laguerre ensembles

\begin{enumerate}
\item
In all cases (a)-(d)  (in case (d) only if $n-p(n) = o(n)$) in Theorem \ref{theoremmodL}, the sequence $(\frac{X_i^L(n)}{2 \log n/ \beta})_n$ satisfies a large
deviation principle with speed $\log n$ and rate function $\frac{x^2}{2}$.
%If $p(n)= p$ for a fixed $p \in \na$, the sequence $ \bigl( \log \bigl( \det W_{n,p}^{L, \beta} \bigr) - \mu_4^L(d) \bigr)_n$ satisfies a LDP with speed $p \, n$ and rate function $I(x)=F(x)$ in \eqref{rate}.

\item
In all cases (a)-(d)  (in case (d) only if $n-p(n) = o(n)$) in Theorem \ref{theoremmodL}, for any sequence $(a_n)_n$ with $a_n = o(\sqrt{\log n})$,
the sequence $ Y_n^{L, \beta} := \frac{X_i^L(n)}{a_n \sqrt{2 \log n/ \beta}}$ satisfies a large deviation principle with speed $a_n^2$ and rate function $\frac{x^2}{2}$.
\end{enumerate}
\end{corollary}

\begin{proof}
(a): In all cases we apply the mod-$\phi$ convergence of Theorem \ref{theoremmodL}, here Theorem \ref{pLDP}, combined with the Theorem of G\"artner-Ellis, see \cite[Theorem 2.3.6]{Dembo/Zeitouni:LargeDeviations}.

(b) Now we apply Theorem \ref{exCLT} for $y = t \, a_n$ with $t \in \re$. Hence we have
$$
P ( Y_n^{L, \beta} \geq t  ) = P (N(0,1) \geq t \, a_n) (1 +o(1)).
$$ 
A famous result, called  Mill's ratio, tells us:
$$
\frac{1}{\sqrt{2 \pi}} \frac{t a_n}{1 + (t a_n)^2} \exp \bigl( - \frac{(t a_n)^2}{2} \bigr) \leq P (N(0,1) \geq t \, a_n) \leq 
\frac{1}{\sqrt{2 \pi}} \frac{1}{t a_n} \exp \bigl( - \frac{(t a_n)^2}{2} \bigr).
$$
Now we take the logarithm and apply the condition $a_n = o( \sqrt{ \log n} )$. To obtain the full principle of large deviations,
proceed as in \cite[Proof of Theorem 1.4]{ERS:2015} applying Theorem 4.1.11 in \cite{Dembo/Zeitouni:LargeDeviations}.
\end{proof}

\begin{corollary}
In case (d), now we assume that $\frac{p(n)}{n} \to c \in [0,1)$ as $n \to \infty$. Then we obtain that
$ \bigl( \log \det W_{n,p(n)}^{L, \beta}  - p(n) (\log n -1) \bigr)$ satisfies an LDP with speed $p(n) \,n$ and
rate function $I$ which is the Legendre-Fenchel transform of
$$
\eta(z) = z  \log 2 + \biggl( z + \frac{\beta}{2} (1-\frac{c}{2}) \biggl) \log \biggl( \frac{\beta}{2} ( 1 - c) +  z  \biggr) - \frac{\beta}{2} (1 -\frac{c}{2}) \log \biggl( \frac{\beta}{2} ( 1 -c)\biggr).
$$
If $c=0$, the rate function $I$ is given by 
\begin{equation} \label{rate}
I(x) = \sup_{a \in \re}  \bigl( a x - \eta(a) \bigr) = \exp( x - \log 2 -1)  - \frac{\beta}{2} x + \frac{\beta}{2} \log \beta,
\end{equation}
$\eta$ given by \eqref{stable1}.
%\eqref{rate}.
\end{corollary}

\begin{proof}
Assume that $\frac{p(n)}{n} \to c \in [0,1)$. We apply \eqref{Stirling} 
for $\Gamma\bigl(\frac{\beta}{2} (n-p(n)+k) + n z \bigr)$ and $\Gamma\bigl(\frac{\beta}{2} (n -p(n) +k)\bigr)$,
which leads to
\begin{eqnarray*} 
2^{z \,p(n)\, n} & & \prod_{k=1}^{p(n)} \frac{\Gamma\bigl(
\frac{\beta}{2} (n-p(n)+k) +n z \bigr)}{\Gamma\bigl(\frac{\beta}{2} (n -p(n) +k)\bigr)} \\
& & \hspace{-2cm} \sim2^{z\, p(n) \, n} e^{- z \, p(n) \, n} n^{z\,p(n)\,n} \prod_{k=1}^{p(n)} \frac{
\biggl( \frac{\beta}{2} \bigl( 1 - \frac{p(n)}{n} \bigr) +z \biggr)^{ (n-p(n)) \frac{\beta}{2} +n \, z  + \frac{\beta}{2} k - \frac 12}}{ 
\biggl( \frac{\beta}{2} \bigl( 1 - \frac{p(n)}{n} \bigr) \biggr)^{(n-p(n)) \frac{\beta}{2}  + \frac{\beta}{2} k - \frac 12}} \\
& & \hspace{-2cm}= 2^{z \, p(n) \, n} e^{-z\, p(n)\, n} n^{z\, p(n) \,n }  
\frac{\biggl( \frac{\beta}{2} \bigl( 1 - \frac{p(n)}{n} \bigr) +z \biggr)^{ p(n)(n-p(n)) \frac{\beta}{2} +p(n) \, n \, z  + \frac{\beta p(n)(p(n)+1)}{4} - \frac{p(n)}{2}}}{\biggl( \frac{\beta}{2} \bigl( 1 - \frac{p(n)}{n} \bigr) \biggr)^{p(n)(n-p(n)) \frac{\beta}{2}  + \frac{\beta p(n)(p(n)+1)}{4}  - \frac{p(n)}{2}}}.
\end{eqnarray*}
Hence
\begin{eqnarray} \label{keinmod1}
\log \biggl( \E \biggl[ \biggl( \det W_{n,p(n)}^{L, \beta} \biggr)^{nz} \biggr]  \biggl) &\sim& nz \bigl( p(n) \log n -p(n) \bigr) + n p(n) z \log 2 \nonumber \\
& & \hspace{-3cm} + n p(n) z  \log \biggl( \frac{\beta}{2} \bigl( 1 - \frac{p(n)}{n} \bigr) +  z  \biggr) + n p(n) \frac{\beta}{2} \log \biggl(
1 + \frac{z}{\frac{\beta}{2} \bigl( 1 - \frac{p(n)}{n} \bigr)} \biggr) \nonumber \\
&& \hspace{-3cm}+ \biggl( \frac{\beta p(n)(p(n)+1)}{4} - \frac{p(n)}{2} - \frac{p(n)^2 \beta}{2} \biggr) \log \biggl(
1 + \frac{z}{\frac{\beta}{2} \bigl( 1 - \frac{p(n)}{n} \bigr)} \biggr).
\end{eqnarray}
It follows that
\begin{eqnarray*}
\frac{1}{ n p(n)} \log \E \biggl[ \exp \biggl( z n \bigl( \log \det W_{n,p(n)}^{L, \beta}  - p(n)(\log n -1) \bigr) \biggr) \biggr]  
& \sim & z \log 2 \\
&& \hspace{-8cm} + \bigl( z + \frac{\beta}{2} \bigl) \log \biggl( \frac{\beta}{2} \bigl( 1 - \frac{p(n)}{n} \bigr) +  z  \biggr)
- \frac{\beta}{2} \log \biggl( \frac{\beta}{2} \bigl( 1 - \frac{p(n)}{n} \bigr) \biggr) \\
&& \hspace{-8cm}+ \biggl( \frac{\beta (p(n)+1)}{4n} - \frac{1}{2n} - \frac{p(n) \beta}{2n} \biggr) \log \biggl(
1 + \frac{z}{\frac{\beta}{2} \bigl( 1 - \frac{p(n)}{n} \bigr)} \biggr),
\end{eqnarray*}
and thus
\begin{eqnarray*}
\lim_{n \to \infty} \frac{1}{ n p(n)} \log \E \biggl[ \exp \biggl( z n \bigl( \log \det W_{n,p(n)}^{L, \beta}  - p(n) (\log n -1) \bigr) \biggr) \biggr]  
& = & z \log 2 \\
&& \hspace{-7cm} + \biggl( z + \frac{\beta}{2} (1-\frac{c}{2}) \biggl) \log \biggl( \frac{\beta}{2} ( 1 - c) +  z  \biggr) \\
&& \hspace{-7cm} - \frac{\beta}{2} (1 -\frac{c}{2}) \log \biggl( \frac{\beta}{2} ( 1 -c)\biggr).
\end{eqnarray*}
Now the statement follows with \cite[Theorem 2.3.6]{Dembo/Zeitouni:LargeDeviations}.
\end{proof}

\begin{remark} Our calculations in \eqref{keinmod1} show that there is no hope to observe mod-$\phi$ convergence
for the sequence $ \bigl( n\log \det W_{n,p(n)}^{L, \beta}  - n (\log n -1) \bigr)$.
\end{remark}

\begin{theorem} \label{BE}
Rate of convergence for the log-determinant of $\beta$-Laguerre ensembles

In cases (a),(c) and (d) in Theorem \ref{theoremmodL}, we obtain
$$
d_{{\operatorname{Kol}}} \biggl( X_i^L(n) \sqrt{\frac{\beta}{2 \log n}} , N(0,1) \biggr) \leq C(D,1,K_1,2,\frac{1}{\sqrt{2 \pi}}) \biggl( \frac{1}{t_n} \biggr)^{1/2},
$$
where the constant is given by \eqref{constBE} with $D$ and $K_1$ depending only on $\beta$.
\end{theorem}

%\begin{remark}
%It seem to be possible to state a local limit theorem in these cases, applying Theorem 2.6 in \cite{Borgoetal:2017}, see
%Theorem \ref{localT}.
%\end{remark}

\begin{proof}
Case (a): If $p(n)=n$, the statement is Theorem 4.11 in \cite{Borgoetal:2017}. Case (c): Assume that $n-p(n) = c$ for a fixed $c \in \na$.
We adapt the techniques and methods of the proof of Theorem 4.11 in \cite{Borgoetal:2017}.
We have to check that condition (a) in Definition \ref{defzone} is satisfied, which is finding a bound on $|\psi_n(i \xi) -1|$.
Inspired by the proof of Lemma 4.1 in \cite{Borgoetal:2017}, we first consider another representation of 
the limiting function of the mod-Gaussian convergence. 
With the definition of the Barnes $G$-function as the solution of $G(z+1)= G(z) \Gamma(z)$, we obtain
$$
\prod_{k=1}^{p(n)} \frac{\Gamma \bigl( \frac{\beta}{2} (k+c) +z \bigr)}{\Gamma \bigl( \frac{\beta}{2} (k+c)\bigr)} = \frac{ G \bigl( \frac{\beta}{2} (p(n)+c) +z+1 \bigr) \, 
G \bigl( \frac{\beta}{2} (1+c)\bigr)}{ G \bigl( \frac{\beta}{2} (1+c) +z \bigr) \, 
G \bigl( \frac{\beta}{2} (p(n)+c) +1\bigr)}.
$$
To handle the factor which depends on $p(n)$, we use the estimate of Proposition 17 in \cite{AK}, which holds true for $|z| \leq \frac 12 p^{1/6}$ and gives
\begin{equation} \label{AK}
\frac{G(1 + z +p)}{G(1+p)} = (2\pi)^{z/2} e^{-(p+1)z} (1+p)^{z^2/2 + pz} \, S_p(z)
\end{equation}
with $ \log S_p(z) = {\mathcal O} \biggl( \frac{|z| + |z|^2}{p} \biggr)$.
Hence we obtain
\begin{eqnarray*}
L(p(n),c,\beta/2;z) & = & \log G \bigl( \frac{\beta}{2} (1+c)\bigr)  - \log G \bigl( \frac{\beta}{2} (1+c) +z \bigr) \nonumber \\
&& + z \log (\sqrt{2 \pi}) - z \bigl( \frac{\beta c}{2} +1 \bigr) - z \frac{\beta}{2} p(n) + \frac{z^2}{2} \log \bigl( \frac{\beta}{2} (p(n)+c)+1 \bigr) \nonumber \\
&& + z \frac{\beta}{2} (p(n)+c) \log \bigl( \frac{\beta}{2} (p(n)+c) +1 \bigr) + {\mathcal O} \biggl( \frac{|z| + |z|^2}{p(n)} \biggr).
\end{eqnarray*}
We arrive at
\begin{equation} \label{newPhic}
\Phi_{\beta/2}^c(z) := \log G \bigl( \frac{\beta}{2} (1+c)\bigr)  - \log G \bigl( \frac{\beta}{2} (1+c) +z \bigr) + z \bigl(  \log (\sqrt{2 \pi}) - \bigl( \frac{\beta c}{2} +1 \bigr) \bigr).
\end{equation}
Consequently,  we get 
$$
\psi_n(z ) = \exp \big( \Phi_{\beta/2}^c (z) + r_n(z) \bigr),
$$
with
$$
r_n(z) = {\mathcal O} \biggl(  \frac{|z|+|z|^2}{p(n)} \biggr),
$$
as soon as $z \in {\mathcal S}_{\beta/2}$ and $|z| \leq \frac{\beta}{8} p(n)^{1/6}$. 
Therefore, there exists a constant $C$ such that for every $n \geq 1$ and $|\xi| \leq \frac{\beta}{8} p(n)^{1/6}$
$$
|r_n(i \xi)| \leq C \frac{|\xi|+|\xi|^2}{p(n)} \leq  C |\xi| e^{|\xi|},
$$
and further using that $|\xi| \leq \frac{\beta}{8} p(n)^{1/6}$,
$$
|r_n(i \xi)| \leq C.
$$
The constant might depend on $\beta$ and $c$.
With the inequality $|e^z -1| \leq |z| e^{|z|}$ for $z \in \cn$, we have
\begin{eqnarray} \label{tricki}
|\psi_n(i \xi) -1 | & \leq & |\Phi_{\beta/2}^c (i \xi) + r_n(i \xi)|  e^{ |\Phi_{\beta/2}^c(i \xi) + r_n(i \xi)|} \nonumber \\
& \leq &e^{C} \bigl( 
 |\Phi_{\beta/2}^c (i \xi)| +  C |\xi| e^{|\xi|} \bigr) e^{ |\Phi_{\beta/2}^c(i \xi)|}.
\end{eqnarray}
To achieve our goal, it is sufficient to bound $|\Phi_{\beta/2}^c (i \xi)|$. 
%We now consider another representation of $\Phi_{\beta/2}^c$ (compare
%with the proof of Lemma 4.1 in \cite{Borgoetal:2017}).
%With the definition of the Barnes $G$-function as the solution of $G(z+1)= G(z) \Gamma(z)$ we obtain
%$$
%\prod_{k=1}^{p(n)} \frac{\Gamma \bigl( \frac{\beta}{2} (k+c) +z \bigr)}{\Gamma \bigl( \frac{\beta}{2} (k+c)\bigr)} = \frac{ G \bigl( \frac{\beta}{2} (p(n)+c) +z+1 \bigr) \, 
%G \bigl( \frac{\beta}{2} (1+c)\bigr)}{ G \bigl( \frac{\beta}{2} (1+c) +z \bigr) \, 
%G \bigl( \frac{\beta}{2} (p(n)+c) +1\bigr)}.
%$$
%To manage the factor which depends on $p(n)$ we use the estimation of Proposition 17 in \cite{AK}, which holds true for $|z| \leq \frac 12 p^{1/6}$ and gives
%\begin{equation} \label{AK}
%\frac{G(1 + z +p)}{G(1+p)} = (2\pi)^{z/2} e^{-(p+1)z} (1+p)^{z^2/2 + pz} \, S_p(z)
%\end{equation}
%with $ \log S_p(z) = {\mathcal O} \biggl( \frac{|z| + |z|^2}{p} \biggr)$.
%Hence we obtain
%\begin{eqnarray*}
%L(p(n),c,\beta/2;z) & = & \log G \bigl( \frac{\beta}{2} (1+c)\bigr)  - \log G \bigl( \frac{\beta}{2} (1+c) +z \bigr) \nonumber \\
%&& + z \log (\sqrt{2 \pi}) - z \bigl( \frac{\beta c}{2} +1 \bigr) - z \frac{\beta}{2} p(n) + \frac{z^2}{2} \log \bigl( \frac{\beta}{2} (p(n)+c)+1 \bigr) \nonumber \\
%&& + z \frac{\beta}{2} (p(n)+c) \log \bigl( \frac{\beta}{2} (p(n)+c) +1 \bigr) + {\mathcal O} \biggl( \frac{|z| + |z|^2}{p(n)} \biggr).
%\end{eqnarray*}
%We arrive at
%$$
%\Phi_{\beta}^c(z) = \log G \bigl( \frac{\beta}{2} (1+c)\bigr)  - \log G \bigl( \frac{\beta}{2} (1+c) +z \bigr) + z \bigl(  \log (\sqrt{2 \pi}) - \bigl( \frac{\beta c}{2} +1 \bigr) \bigr).
%$$
For our purposes in \eqref{newPhic}, we can neglect the constant $ \log G \bigl( \frac{\beta}{2} (1+c)\bigr)$. Hence we try to find a bound for
$$
f(\xi) := - \log G \bigl( \frac{\beta}{2} (1+c) +i \xi \bigr) + i \xi  \bigl(  \log (\sqrt{2 \pi}) - \bigl( \frac{\beta c}{2} +1 \bigr) \bigr).
$$
By Theorem 5.19 in \cite{Rudin:book}, we have $|f(\xi) - f(0) | \leq |\xi| \sup_{t \in (0, \xi)} |f'(t)|$.
Now 
$$
f'(t)= \frac{G' \bigl( \frac{\beta}{2} (1+c) +i t \bigr)}{G \bigl( \frac{\beta}{2} (1+c) +i t \bigr)} + i\,  \bigl(  \log (\sqrt{2 \pi}) - \bigl( \frac{\beta c}{2} +1 \bigr) \bigr).
$$
With \eqref{digamma} and \eqref{boundphi} we have $|\Psi \bigl(\frac{\beta}{2} (1+c) +i t \bigr)| \leq a_1 |t| + a_2$ for some positive constants $a_1, a_2$. Indeed we apply the inequality $|\log(c + it)| \leq |t|$, $c>0$ 
(using $\log(c + i \xi) = \int_0^1 \frac{i \xi}{t (i \xi) +c} dt$, which is valid for $z \in \cn \setminus (-\infty, -c]$ and make $z=i\xi$).
With \eqref{logG} we get
$$
|f'(t)| \leq b_1 t^2 + b_2 |t| + b_3
$$
for some positive constants $b_1,b_2,b_3$. Summarising, we have
\begin{eqnarray} 
|f( i \xi)| & \leq & c_1 |\xi|^3 + c_2 |\xi|^2 + c_3 |\xi| \label{tricki2} \\
&\leq& c_4 |\xi| e^{c_5 |\xi|} \label{tricki3}.
\end{eqnarray}
In addition to that, we use the fact that there exists a $c_6>0$ such that for exery $x \geq 0$
\begin{equation} \label{poly}
 c_1 x^3 + c_2 x^2 + c_3 x  \leq x^3 + c_6.
 \end{equation}
We now consider \eqref{tricki} and successively plug in \eqref{tricki3}, \eqref{tricki2} and \eqref{poly} to obtain
$$
|\psi_n(i \xi) -1 | \leq  K_1 |\xi| \exp (K_2 |\xi|^3).
$$
Therefore, if $n-p(n)=c$, we have checked that the sequence of log-determinants of the $\beta$-Laguerre ensembles converges mod-Gaussian
with zone of control $[-D t_n,D t_n ]$ with $D \leq \frac{1}{4 K_2}$ and index of control $(1,3)$. The result follows considering $\gamma = \min (1, \frac{v-1}{2})$.

\noindent
Case (d): let us assume that $n-p(n) =o(n)$. From Theorem \ref{theoremmodL} we obtain that
$$
\psi_n(z ) = \exp \big( r_n(z) \bigr)
$$
with
$$
r_n(z) = {\mathcal O} \biggl(  \frac{|z|+|z|^2}{p(n)} \biggr) +  {\mathcal O} \biggl(  \frac{|z|+|z|^2}{n-p(n)} \biggr),
$$
as soon as $z \in {\mathcal S}_{\beta/2}$ and $|z| \leq \frac{\beta}{8} \max( p(n), n-p(n))^{1/4}$. 
Therefore, there exists a constant $C$ such that for every $n \geq 1$ and $|\xi| \leq \frac{\beta}{8} \max(p(n),n-p(n))^{1/4}$
$$
|r_n(i \xi)| \leq C \bigl( \frac{|\xi|+|\xi|^2}{p(n)} +  \frac{|\xi|+|\xi|^2}{n-p(n)}  \bigr) \leq  C |\xi| e^{|\xi|}
$$
and further using that $|\xi| \leq \frac{\beta}{8} \max(p(n),n-p(n))^{1/4}$,
$$
|r_n(i \xi)| \leq C.
$$
The constant $C$ might depend on $\beta$. Therefore it is enough to apply the same trick as in the case $n-p(n)=c$ to obtain
$|\psi_n(i\xi)| \leq K_1 |\xi| e^{K_2 |\xi|^3}$. Thus, if $n-p(n)=c(n)$, we have checked that the sequence of log-determinants of the $\beta$-Laguerre ensembles converges mod-Gaussian with zone of control $[-D t_n ,D t_n]$ for some $D>0$ with index of control $(1,3)$. The result follows considering $\gamma = \min (1, \frac{v-1}{2})$.
%Indead using the inequality $|\log(1 + it)| \leq |t|$
%(using $\log(1 + i \xi) = \int_0^1 \frac{i \xi}{t (i \xi) +1} dt$), it is a simple exercise to obtain (compare with \eqref{Phid})
%$$
%\big| \sum_{j=1}^4 U_j(c, \beta/2; i\xi) \big| \leq C(\beta,c)  \bigl( |\xi| + |\xi|^2 + |\xi|^3 \bigr)
%$$
%with some constant depending on $\beta$ and $c$.
%But it is not obvious how to bound $T_5(c, \beta/2; i \xi)$.
\end{proof}

\begin{theorem} \label{local}
Local limit theorem for log-determinants of $\beta$-Laguerre ensembles

In case (a),(c) and (d) in Theorem \ref{theoremmodL}, we obtain for any $\mu \in (0, \frac 32)$
$$
P \biggl( \frac{X_i^L(n)}{\sqrt{\frac{\beta}{2 \log n}}} - x \in (\log n)^{-\mu} B \biggr) \simeq (\log n)^{- \mu} \frac{e^{-x^2/2}}{\sqrt{2 \pi}} m(B).
$$
\end{theorem}

\begin{proof} The result follows immediately from the proof of Theorem \ref{BE}.
\end{proof}

\begin{remark}
The speed of convergence and the statement of a local theorem in the case of mod-stable convergence $p(n)=p$ is not available. The reason for this is that
the estimate \eqref{AK} is not sharp enough with respect to the order of the approximation-error $\log S_p(n z)$.
\end{remark}
%\begin{theorem}
%Rate of convergence for log-determinant of $\beta$-Laguerre ensembles if $p(N)=p$ for a fixed $p \in \na$.
%In case (e) in Theorem \ref{theoremmodL} we obtain
%$$
%Y_n^L(\beta,p,c) := \frac 1p \log \bigl( \det W_{n,p}^{L, \beta} \bigr) + \frac{2c}{\pi} \log (pn).
%$$
%We observe the following speed of convergence:
%$$
%d_{{\operatorname{Kol}}} \biggl( Y_n^L(\beta,p,c)  , \phi_{c,1,-1} \biggr) \leq C(D,1,K_1,1,c) \, \frac{1}{p \, n},
%$$
%where the constant is given by \eqref{constBE} with $D$ and $K_1$ depending only on $\beta$ and $p$.
%\end{theorem}

\subsection{$\beta$-Jacobi ensembles}
A direct consequence of Theorem \ref{cumulant} is mod-$\phi$ convergence for the shifted log-determinants of the
$\beta$-Jacobi ensembles. 
Remember that by \eqref{MellinJ} we have
\begin{eqnarray*}
\log \E \biggl[ \exp \bigl( z \, \log \bigl( \det W_{p(n),n_1,n_2}^{J, \beta} \bigr) \bigr) \bigg] &=& L(p(n), n_1-p(n), \beta/2;z)
\\ && \hspace{2cm}  - L(p(n), n_1+n_2-p(n), \beta/2;z).
\end{eqnarray*}
Although in subsection \ref{subLag} the asymptotic behaviour of $L(p(n), n_1-p(n), \beta/2;z)$ has been analysed completely,
we have to add a case by case analysis of $L(p(n), n_1+n_2-p(n), \beta/2;z)$:

\begin{proposition} \label{cumulant2}
$L(p(n), n_1+n_2-p(n), \beta/2;z)$ defined in \eqref{L} satisfies the following asymptotic expansion
locally uniformly on ${\mathcal S}_{\beta/2}$:
$$
L(p(n),n_1+n_2-p(n), \beta/2;z) = z \mu(p(n),n_1,n_2) + \frac{z^2}{\beta} \log \biggl( \frac{n_1+n_2}{n_1+n_2-p(n)} \biggr) +o(1),
$$ 
where 
\begin{eqnarray} \label{muJ}
\mu(p(n),n_1,n_2) &:=& p(n) \log (\beta/2) + (n_1+n_2) \log  (n_1+n_2) \nonumber \\
& - & (n_1+n_2-p(n)) \log (n_1+n_2-p(n)) \nonumber \\
& +&  \frac 12 \log (2 \pi) + \biggl( \frac 12 - \frac{1}{\beta} \biggr) \log \biggl( \frac{n_1 + n_2}{n_1 + n_2 -p(n)} \biggr).
\end{eqnarray}
\end{proposition}

\begin{proof}
The proof follows after an easy calculation. Therefore, we will use \eqref{T1rn}, \eqref{T2rn} as well  as the
expansion of
$$
\sum_{k=0}^{p(n) -1} f_{n_1+n_2-p(n)}(k)
$$
as in \eqref{T3Nullexp}, together with Theorem \ref{MAIN}.
\end{proof}

Now the combination of Theorem \ref{cumulant} and Proposition \ref{cumulant2} leads to the following mod-$\phi$ structure for the log-determinant of the Jacobi-ensembles.

\begin{theorem} \label{theoremmodJ}
Mod-$\phi$ convergence for the log-determinant of $\beta$-Jacobi ensembles

We observe the following results on ${\mathcal S}_{1/2}$:
\bigskip

\hspace{-2cm}{
\begin{tabular}{|c|c|c|c|c|} \hline
condition & centred version of & mod-$\phi$ & $t_n$ & limiting function \\ \hline
$p(n) =n_1$ & $\log (\det W_{n_1,n_1,n_2}^{J,\beta})$ & mod-$N(0,1)$ & $\frac{2}{\beta} \log \biggl( \frac{n_1 \, n_2}{n_1 + n_2} \biggr)$ & $\exp (\Phi_{\beta/2}(z))$ \\ \hline
$n_1-p(n) \to 0$ & $\log (\det W_{p(n),n_1,n_2}^{J,\beta})$ & mod-$N(0,1)$ & $\frac{2}{\beta} \log \biggl( \frac{n_1(n_1-p(n)+n_2)}{n_1+n_2} \biggr)$ & $\exp (\Phi_{\beta/2}(z))$ \\ \hline
$n_1-p(n)=c$ & $\log (\det W_{p(n),n_1,n_2}^{J,\beta})$ & mod-$N(0,1)$ & $\frac{2}{\beta} \log \biggl( \frac{n_1(c+n_2)}{n_1+n_2} \biggr)$ & $\exp (\Phi_{\beta/2}^c(z))$ \\ \hline
$n_1-p(n) = o \bigl(\frac{n_1 \, n_2}{n_1 + n_2}\bigr)$ & $\log (\det W_{p(n),n_1,n_2}^{J,\beta})$ & mod-$N(0,1)$ & 
$\frac{2}{\beta} \log \biggl(\frac{n_1(n_1+n_2-p(n))}{(n_1-p(n))(n_1+n_2)} \biggr)$
& $1$ \\ \hline
$p(n)=p $ & $ \log (\det W_{p,n_1,n_2}^{J,\beta})$ & no mod-$\phi$ & --- & --- \\ \hline
\end{tabular}}
\bigskip

\noindent
Here the limit is the limit as $n_1$ and $n_2$ tend to $\infty$ simultaneously.
Moreover, we obtain that $\mu_1^J= \mu_1(n_1,n_1) - \mu(n_1,n_1,n_2)$ (defined in \eqref{mupn} and \eqref{muJ} respectively) to be the
expectation if $p(n)=n_1$. If $n_1 -p(n) \to 0$ as $n \to \infty$, it is $\mu_2^J= \mu_1(p(n),n_1) - \mu(p(n),n_1,n_2)$.
If $n_1-p(n) = c$ for a fixed $c \in \na$, 
the expectation is $\mu_3^J = \mu_2(p(n),n_1) - \mu(p(n),n_1,n_2)$ (defined in \eqref{munew} and \eqref{muJ} respectively).
Finally, for $n_1-p(n) =o \bigl(\frac{n_1 \, n_2}{n_1 + n_2}\bigr)$, we obtain that
$\mu_4^J = \mu_3(p(n),n_1) - \mu(p(n),n_1,n_2)$ (defined in \eqref{munew2} and \eqref{muJ} respectively) is the correct expectation.
\end{theorem}

\begin{proof}
If $n_1 -p(n) \to \infty$ we have to assume in addition that $n_1-p(n) = o \bigl(\frac{n_1 \, n_2}{n_1 + n_2}\bigr)$ 
to ensure that the parameter sequence $(t_n)_n$ is increasing. The results follow immediately from Theorem \ref{cumulant} and Proposition \ref{cumulant2}.
\end{proof}

\begin{theorem} \label{modJ2}
We fix $\tau_1, \tau_2>0$ and assume that $n_1 = \lfloor n \tau_1 \rfloor$ and $n_2 = \lfloor n \tau_2 \rfloor$.
In this regime we obtain for the centred version of  $\log (\det W_{p(n),  \lfloor n \tau_1 \rfloor ,  \lfloor n \tau_2 \rfloor}^{J,\beta})$ on ${\mathcal S}_{1/2}$: 
\bigskip

\hspace{-1cm}{
\begin{tabular}{|c|c|c|c|} \hline
condition & mod-$\phi$ & $t_n$ & limiting function \\ \hline
$p(n) =  \lfloor n \tau_1 \rfloor$  & mod-$N(0,1)$ & $\frac{2}{\beta} \log n$ & $\exp \biggl( \frac{z^2}{\beta} \log \bigl( \frac{ \tau_1 \tau_2}{\tau_1 + \tau_2} \bigr) + \Phi_{\beta/2}(z) \biggr)$ \\ \hline
$ \lfloor n \tau_1 \rfloor -p(n) \to 0$ & mod-$N(0,1)$ & $\frac{2}{\beta} \log \biggl( n + \frac{n \tau_1-p(n)}{\tau_2} \biggr)$ & 
$\exp \biggl( \frac{z^2}{\beta} \log \bigl( \frac{ \tau_1 \tau_2}{\tau_1 + \tau_2} \bigr)+ \Phi_{\beta/2}(z) \biggr)$ \\ \hline
$\lfloor n \tau_1 \rfloor-p(n)=c$ & mod-$N(0,1)$ & $\frac{2}{\beta} \log \biggl( n + \frac{c}{\tau_2} \biggr)$ & 
$\exp \biggl( \frac{z^2}{\beta} \log \bigl( \frac{ \tau_1 \tau_2}{\tau_1 + \tau_2} \bigr)+ \Phi_{\beta/2}^c(z) \biggr)$ \\ \hline
$\lfloor n \tau_1 \rfloor-p(n) = o \bigl(n \bigr)$ & mod-$N(0,1)$ & 
$\frac{2}{\beta} \log \biggl( \frac{n}{(n \tau_1-p(n))} + \frac{1}{\tau_2} \biggr)$
& $\frac{ \tau_1 \tau_2}{\tau_1 + \tau_2} \exp \bigl( \frac{z^2}{\beta} \bigr) $ \\ \hline
%$p(n)=p $ & $ \log (\det W_{p,n_1,n_2}^{J,\beta})$ & no mod-$\phi$ & --- & --- \\ \hline
\end{tabular}}
\bigskip

{\bf Case $p(n)=p$ for a fixed $p \in \na$:} The centred version of the sequence $\log \bigl( \det W_{p,\lfloor n \tau_1 \rfloor,\lfloor n \tau_2 \rfloor}^{J, \beta} \bigr)$ does not converge in the sense of mod-$\phi$ convergence. However, the sequence $(n \log \bigl( \det W_{p, \lfloor n \tau_1 \rfloor, \lfloor n \tau_2 \rfloor}^{J, \beta} \bigr) \bigr)_{n}$ converges {\bf mod-$\phi$} on $i \, \re$ with parameter $t_{n} = p \, n$ and limiting function 
$$
\psi(z) = \biggl( \frac{\tau_1 + \tau_2}{\tau_1} \frac{\tau_1 \beta +2z}{(\tau_1+\tau_2)\beta+2z} \biggr)^{-\frac{\beta(p-1)p}{4} - \frac p2}.
$$
Here $\phi$ is such that the L\'evy exponent is given by
\begin{eqnarray*} \label{stable2}
\eta(z) = \log \int_{\re} e^{zx} \phi(dx) &=& \frac{\beta}{2} (\tau_1 +\tau_2) \log \biggl( \frac{\beta}{2} (\tau_1+\tau_2) \biggr) -  \frac{\beta}{2} \tau_1 \log \biggl( \frac{\beta}{2} \tau_1 \biggr) \\
& &\hspace{-2cm} + \bigl( z +\tau_1 \frac{\beta}{2} \bigr) \log \bigl( z +\tau_1 \frac{\beta}{2} \bigr) 
 -  \bigl( z + (\tau_1 + \tau_2) \frac{\beta}{2} \bigr) \log  \bigl( z + (\tau_1 + \tau_2) \frac{\beta}{2} \bigr).
\end{eqnarray*}
\end{theorem}

\begin{proof}
First notice that the choices $n_1 = \lfloor n \tau_1 \rfloor$ and $n_2 = \lfloor n \tau_2 \rfloor$ lead to the extra summand 
$$
\frac{z^2}{\beta} \log \bigl( \frac{ \tau_1 \tau_2}{\tau_1 + \tau_2} \bigr)
$$
in the limiting function.
The only case we have to consider is the case  $p(n)=p$ for a fixed $p \in \na$. Here we apply Stirling's formula
like in the proof of Theorem \ref{theoremmodL}, case (e), see \eqref{StirAnw}. We apply the formula to
$L(p, \lfloor n \tau_1 \rfloor - p, \beta/2; n z)$ as well as $L(p, \lfloor n \tau_1 \rfloor+\lfloor n \tau_2 \rfloor-p, \beta/2;nz)$.
The calculations are similar to \eqref{StirAnw} and are left to the reader.
\end{proof}
We do not formulate the corresponding extended central limit theorems, precise deviations, large and moderate deviation principles 
and Berry-Esseen bounds for the log-determinants of the Jacobi ensemble, because these results can be stated as in Theorems \ref{theoremmodL}, \ref{exCLT}, \ref{pLDP}, Corollary \ref{LDP}  and Theorem \ref{BE}.

\subsection{Ginibre ensembles}
As a corollary of Theorem \ref{cumulant}, case (a), we obtain for the log-determinants of the Ginibre ensemble with \eqref{MellinG}
\begin{eqnarray*} 
\log \E \biggl[ \biggl( \det W_{n}^{G, \beta} \biggr)^z \biggr] &= & z \bigl( \frac{n}{2} \log \bigl( \frac{2}{\beta} \bigr) \bigr) + L(n,0,\beta/2;z) \\
& =& z \mu^G(n) + \frac{z^2}{\beta} \log n + \Phi_{\beta/2}(z) + o(1),
\end{eqnarray*}
where $\mu^G(n)= \frac{n}{2} \log \bigl( \frac{2}{\beta} \bigr) + \mu_1(n,n)$, and $\mu_1(n,n)$ is defined in \eqref{mupn}.
Hence we obtain mod-Gaussian convergence, an extended central limit theorem, precise deviations, large and moderate deviation principles and Berry-Esseen bounds
for the log-determinants of the Ginibre ensemble as in case (a) in Theorems \ref{theoremmodL}, \ref{exCLT}, \ref{pLDP}, Corollary \ref{LDP} and Theorem \ref{BE}. Notice that these results follow directly from the results in \cite{Borgoetal:2017}.

\subsection{Ensembles in mesoscopic physics}
Let us first consider the {\it chiral ensembles}. Here we obtained in \eqref{Mellin7ch} that
$$ 
\E \biggl[ \biggl( \det W_{n,p(n)}^{\beta, \mu_{\operatorname{chiral}}} \biggr)^z \biggr] = L(p(n), n-p(n), \beta/2; \frac{z+1}{2})
$$
for $\beta \in \{1,2,4\}$. Hence we obtain mod-Gaussian convergence, an extended central limit theorem, precise deviations, large and moderate deviation principles and Berry-Esseen bounds for the logarithm of the product of the positive eigenvalues of the chiral ensemble as in Theorems \ref{theoremmodL}, \ref{exCLT}, \ref{pLDP}, Corollary \ref{LDP} and Theorem \ref{BE}.
\medskip

Next we consider one of the Bogoliubov-de Gennes ensembles. We proved in \eqref{Mellin7BdG} that
$$
\E \biggl[ \biggl( \det W_{n,n}^{1,1} \biggr)^z \biggr] = L(n,1,1/2; \frac{z+1}{2}) = \prod_{k=1}^{n} \frac{\Gamma \bigl( \frac{1}{2} (k+1) +\frac{z+1}{2} \bigr)}{\Gamma \bigl( \frac{1}{2} (k+1)\bigr)}.
$$
With the same steps as in the proof of Theorem \ref{BE} we find
$$
\E \biggl[ \biggl( \det W_{n,n}^{1,1} \biggr)^z \biggr] = \frac{ G \bigl( \frac{1}{2} (n+1) +\frac{z+1}{2} +1 \bigr) \, 
G \bigl( 1 \bigr)}{ G \bigl( 1 +\frac{z+1}{2} \bigr) \, 
G \bigl( \frac{1}{2} (n+1) +1\bigr)}.
$$
Now we adapt the calculations in the proof of Theorem \ref{BE} and obtain mod-Gaussian convergence with $t_n = \log \bigl( \frac 12 (n+1) +1) \sim \log n$
and limiting function $\exp \bigl( \Phi_{1/2}^1 \bigl( \frac{z+1}{2} \bigr) \bigr)$, where $\Phi_{1/2}^1(\cdot)$ is defined in \eqref{newPhic}.

An extended central limit theorem, precise deviations, large and moderate deviation principles and Berry-Esseen bounds
for the log-determinants of the Bogoliubov-de Gennes ensembles are consequences of case (c) in Theorems \ref{exCLT}, \ref{pLDP}, Corollary \ref{LDP} and Theorem \ref{BE}.

\subsection{Trace-fixed GUE}

Recall that in \eqref{MellinGUEft}, we obtained the identity
$$
\E \biggl[ \big| n^{n/2}  \det W_n^{H,ft} |^z \biggr] = \prod_{k=1}^{n} \frac{\Gamma\bigl(
\frac{z+1}{2} + \lfloor \frac k2 \rfloor \bigr)}{ \Gamma \bigl( \frac{1}{2} + \lfloor \frac k2 \rfloor\bigr) } \,
\biggl( \prod_{k=1}^{n} \frac{\Gamma\bigl(
\frac{z}{2} + \frac{n}{2} +\frac{k-1}{n} \bigr)}{ \Gamma \bigl( \frac{n}{2} + \frac{k-1}{n} \bigr) } \biggr)^{-1}. 
$$
Let us consider the case where $n$ is an odd number. We leave it to the reader to check that $n$ even leads to the same asymptotics.
From Lemma 4.1 in \cite{Borgoetal:2017} we know that locally uniformly on the band ${\mathcal S}_1$
$$
\log \prod_{k=1}^{n} \frac{\Gamma\bigl(
\frac{z+1}{2} + \lfloor \frac k2 \rfloor \bigr)}{ \Gamma \bigl( \frac{1}{2} + \lfloor \frac k2 \rfloor\bigr) } = z \mu_n^H + \frac{z^2}{4} \log \bigl( \frac n2 \bigr) + \Phi^H(z) +o(1)
$$
with $\mu_n^H = \frac 12 \log (2 \pi) - n  + \frac n2 \log n$ (we have to adapt the result in \cite{Borgoetal:2017} by a summand $-\frac n2$). The function $\Phi^H(z)$
is defined as
$$
\Phi^H(z) = \log \biggl( \frac{ \Gamma \bigl( \frac 12 \bigr) \, G \bigl( \frac 12 \bigr)^2}{  \Gamma \bigl( \frac{z+1}{2} \bigr) \, G \bigl( \frac{z+1}{2} \bigr)^2} \biggr).
$$
Moreover, we have
$$
\prod_{k=1}^{n} \frac{\Gamma \bigl( \frac{z}{2} + \frac{n}{2} +\frac{k-1}{n} \bigr)}{ \Gamma \bigl( \frac{n}{2} + \frac{k-1}{n} \bigr)} =
\frac{G\bigl(\frac z2 +1+\bigl( \frac n2 +1 - \frac 1n \bigr)\bigr) G\bigl(\frac n2-1+1\bigr)}{G\bigl(\frac z2 + 1 + \frac n2 -1 \bigr)G \bigl( \bigl( \frac n2 +1 - \frac 1n \bigr) +1 \bigr)}.
$$
Now we apply \eqref{AK} twice for $|z| \leq \frac 12 \bigl( \frac n2 \bigr)^{1/6}$ and $z \in {\mathcal S}_1$ to obtain
$$
\log \prod_{k=1}^{n} \frac{\Gamma\bigl(\frac{z}{2} + \frac{n}{2} +\frac{k-1}{n} \bigr)}{ \Gamma \bigl( \frac{n}{2} + \frac{k-1}{n} \bigr)}
= z \,f(n) + \frac{z^2}{4} \log \bigl( 1 + \frac 4n - \frac{2}{n^2} \bigr) +{\mathcal O} \biggl( \frac{|z| + |z|^2}{n} \biggr)
$$
with
$$
f(n) = \frac{1}{2n} - \frac 34 + \frac 12 \bigl( \frac n2 +1 - \frac 1n \bigr) \log \bigl( \frac n2 +2 - \frac 1n \bigr) - 
\frac 12 \bigl( \frac n2 -1\bigr) \log \bigl( \frac n2 \bigr).
$$
Summarising, we observe mod-Gaussian convergence with $t_n= \frac 12 \log \bigl( \frac n2 \bigr) - \frac 12  \log \bigl( 1 + \frac 4n - \frac{2}{n^2} \bigr)$, limiting function $\exp (\Psi^H(z))$ and expectation of order $\mu_n^H + f(n)$.
We skip the formulation of an extended central limit theorem, precise deviations, large and moderate deviation principles and Berry-Esseen bounds for the sum of the log-eigenvalues in the GUE fixed-trace ensemble. It can be stated
similarily to the statements in Theorems \ref{exCLT}, \ref{pLDP}, Corollary \ref{LDP} and Theorem \ref{BE}. 

\section{Gram ensembles, random parallelotopes and simplices}

To be able to proof the results for the log-volume of random parallelotopes and random simplices, we will prove the following result:

\begin{proposition} \label{Binetapp}
Let $m(n,\nu)$ be a sequence in $n$, with values in $\re$, where $\nu>0$ is a real number.
Assume that $m(n,\nu)$ is increasing in $n$ with $m(n,\nu) \leq c(\nu) n$ with a constant $c(\nu)$ depending only on $\nu$.
Then we have
\begin{eqnarray} \label{B1}
\log \Gamma (m(n,\nu) +z) - \log \Gamma(m(n,\nu)) &= &z \biggl( \log m(n,\nu) - \frac{1}{2 m(n,\nu)} \biggr) \nonumber \\
&+& \frac{z^2}{m(n,\nu)}
+ {\mathcal O} \biggl( \frac{|z| + |z|^2 + |z|^3}{m(n,\nu)} \biggr)
\end{eqnarray}
for any $z \in {\mathcal S}_{c(\nu)}$ with $|z| < \bigl( c(\nu)/4 \bigr) m(n,\nu)^{1/6}$.
\end{proposition}

\begin{proof}
This is an easy application of the first Binet's formula \eqref{Binet1.2} as well as expanding the logarithm and using
the estimate in the proof of \eqref{T1rn}.
\end{proof}

For the log-volume of the parallelotope spanned by random points $X_1, \ldots, X_{p(n)}$ as well as for
the log-volume of the simplex with vertices $X_1, \ldots, X_{p(n)+1}$ (see Section 1.4), we obtain
the following results:

\begin{theorem}[Gaussian model, Parallelotope and Simplex]
The logarithm of the $p(n)$-dimensional volume of the parallelotope (see \eqref{MellinPa1})
satisfies the results of Theorem \ref{theoremmodL}, where one has to replace $z$ by $z/2$. 

The logarithm $\log (p(n)! V S_{n,p(n)})$ of the $p(n)$-dimensional volume of the simplex satisfies the same
mod-$\phi$ convergence properties. We only have to change the expectations: add $\log (p(n)+1)$ to the expectations
in the parallelotope case; see \eqref{MellinSim1}.
\end{theorem}

\begin{proof}
The case of a paralleltope follows from \eqref{MellinPa1} and Theorem \ref{theoremmodL}. For the case of a simplex, see \eqref{MellinSim2}
and Theorem \ref{theoremmodL}.
\end{proof}

\begin{theorem}[Beta and spherical model, Parallelotope] \label{BSP}
The logarithm of the $p(n)$-dimensional volume of the parallelotope in the beta model (see \eqref{MellinPa2})
and in the spherical model (see \eqref{MellinPa4}) satisfy the following results on ${\mathcal S}_{1/2}$: 
\bigskip

\hspace{-1cm}{
\begin{tabular}{|c|c|c|c|c|} \hline
condition & centred version of & mod-$\phi$ & $t_n$ & limiting function \\ \hline
$p(n) =n$  & $\log (n! V P_{n,n})$ &mod-$N(0,1)$ & $\frac{2}{\beta} \log n$ & $\exp (\Phi_{1/2}(z) - \frac{z^2}{2})$ \\ \hline
$n-p(n) \to 0$  & $\log (p(n)! V P_{n,p(n)} )$ &mod-$N(0,1)$ & $\frac{2}{\beta} \log n$ & $\exp (\Phi_{1/2}(z) -\frac{z^2}{2})$ \\ \hline
$n-p(n)=c$  & $\log (p(n)! V P_{n,p(n)})$ &mod-$N(0,1)$ & $\frac{2}{\beta} \log n$ & $\exp (\Phi_{1/2}^c(z) - \frac{z^2}{2})$ \\ \hline
$n-p(n) =o(n)$ & $\log (p(n)! V P_{n,p(n)})$& mod-$N(0,1)$ & $\frac{2}{\beta} \log \bigl( \frac{n}{n-p(n)} \bigr)$ & $\exp(-z^2/2)$ \\ \hline
$p(n)=p $ & $ n\, \log (p! V P_{n,p})$ & mod-$\phi$ in $i \, \re$ & $p\, n$ & $ \bigl( 1 + z \bigr)^{-\frac{(p-1)p}{4} - \frac p2 - p \bigl(\frac{\nu}{2} - \frac 12 \bigr)}$ \\ \hline
\end{tabular}}
\bigskip

\noindent
The corresponding expectations of the log-volumes are $\mu_1(n,n) +\frac{n}{n+\nu}-n \log (\frac{n+\nu}{2})$ (with $\mu_1(n,n)$
defined in \eqref{mupn}) in the case $p(n)=n$. If $n -p(n) \to 0$ as $n \to \infty$, it is $\mu_1(p(n),n) +\frac{p(n)}{n+\nu}-p(n) \log (\frac{n+\nu}{2})$.
If $n-p(n) = c$ for a fixed $c \in \na$, 
the expectation is $\mu_2(p(n),n) + \frac{p(n)}{n+\nu} - p(n) \log (\frac{n+\nu}{2})$.
Finally for $n-p(n) =o(n)$, we have
$\mu_3(p(n),n) + \frac{p(n)}{n+\nu} - p(n) \log (\frac{n+\nu}{2})$. Here $\mu_2(p(n),n)$ and $\mu_3(p(n),n)$ are defined in \eqref{munew} and
\eqref{munew2}.% is the correct expectation. 

\noindent
If $p(n)=p$, we obtain the mod-$\phi$ on $i \, \re$ result with L\'evy exponent
$$
\eta(z) = \frac z2 \log (\frac 12) 
$$
and expectation $\mu_4 = p \log 2$. The case $\nu=0$ leads to the spherical model.
\end{theorem}

\begin{proof}
The proof is the same as the proof of Lemma 4.2 in \cite{Borgoetal:2017}. We apply Proposition \ref{Binetapp} with $m(n,\nu)= \frac{n +\nu}{2}$. We get
\begin{eqnarray*}
p(n) \bigl(- \log \Gamma (m(n,\nu) +z/2) + \log \Gamma(m(n,\nu))) &=& \frac z2 \, p(n) \biggl( \frac{1}{n+\nu} - \log \bigl( \frac{n + \nu}{2} \bigr) \biggr) \\
&&\hspace{-1cm} - \frac{p(n)}{2(n + \nu)} z^2 + {\mathcal O} \biggl( \frac{|z| + |z|^2 + |z|^3}{n} \biggr).
\end{eqnarray*}
The results follow with \eqref{MellinPa2} and \eqref{MellinPa4} together with Theorem \ref{cumulant} and \ref{theoremmodL}.
Finally, we have to consider the case $p(n)=p$. From \eqref{StirAnw} we observe that
\begin{eqnarray} \label{Stir3} 
L(p,n-p,\frac 12;\frac{nz}{2})
 &\sim& n \frac z2 \bigl( p \, \log n -p\bigr)
+  p\, n \biggl( \frac{1}{2}  \log 2 + \bigl(\frac z2 + \frac{1}{2} \bigr) \log  \bigl(\frac z2 + \frac{1}{2} \bigr) \biggr) \nonumber\\
&+&  \biggl( -\frac{p(p-1)}{4} - \frac p2 \biggr) \log \bigl( 1 +z \bigr). 
\end{eqnarray}
With \eqref{Stirling}  we have
\begin{eqnarray*}
p \log \Gamma \bigl( \frac{n + \nu}{2} \bigr)  - p \log \Gamma \bigl( \frac{n + \nu}{2}  + \frac{nz}{2}\bigr)  & \sim &  n \frac z2 \bigl( p- p \, \log \frac n2 \bigr) \nonumber \\
& & \hspace{-7cm}+ p\, n \biggl( -  \bigl(\frac z2 + \frac{1}{2} \bigr) \log  \bigl(1+z \bigr) \biggr) 
-  p \biggl(  \frac{\nu}{2} - \frac 12 \biggr) \log \bigl( 1 +z \bigr). 
\end{eqnarray*}
Hence the statement is proven.
\end{proof}

\begin{remark}
Recall that $\Phi_{1/2}(\cdot)$ has the following expression, see \cite[Lemma 7.1 (2)]{Borgoetal:2017}: 
$$
\Phi_{1/2}(z) = z \biggl( - \frac 12 \log \frac 12 + \frac 12 \log (2 \pi) \biggr) - \frac 12 \log G(1 + 2 z) - \frac 12 \biggl( \log \Gamma(\frac 12) - \log \Gamma (\frac 12 +z) \biggr).
$$
\end{remark}
\bigskip

\begin{theorem}[Beta and spherical model, Simplex]
The logarithm of the $p(n)$-dimensional volume of the simplex in the beta model (see \eqref{MellinSim2})
and in the spherical model (see \eqref{MellinSim4}) satisfies the following results on ${\mathcal S}_{1/2}$: 
\bigskip

\hspace{-1cm}{
\begin{tabular}{|c|c|c|c|c|} \hline
condition & centred version of & mod-$\phi$ & $t_n$ & limiting function \\ \hline
$p(n) =n$  & $\log (n! V S_{n,n})$ &mod-$N(0,1)$ & $\frac{2}{\beta} \log n$ & $\exp (\Phi_{1/2}(z) - \frac{z^2}{2})$ \\ \hline
$n-p(n) \to 0$  & $\log (p(n)! V S_{n,p(n)} )$ &mod-$N(0,1)$ & $\frac{2}{\beta} \log n$ & $\exp (\Phi_{1/2}(z) -\frac{z^2}{2})$ \\ \hline
$n-p(n)=c$  & $\log (p(n)! V S_{n,p(n)})$ &mod-$N(0,1)$ & $\frac{2}{\beta} \log n$ & $\exp (\Phi_{1/2}^c(z) - \frac{z^2}{2})$ \\ \hline
$n-p(n) =o(n)$ & $\log (p(n)! V S_{n,p(n)})$& mod-$N(0,1)$ & $\frac{2}{\beta} \log \bigl( \frac{n}{n-p(n)} \bigr)$ & $\exp(-z^2/2)$ \\ \hline
$p(n)=p $ & $ n\, \log (p! V S_{n,p})$ & mod-$\phi$ in $i \, \re$ & $p\, n$ & $ \psi^p(z)$ \\ \hline
\end{tabular}}
\bigskip

\noindent
The corresponding expectations of the log-volumes are $\mu_1(n,n) + \frac{n+1}{n+\nu}-(n+1) \log (\frac{n+\nu}{2})
+ 2 \log n - \log 2)$ (with $\mu_1(n,n)$
defined in \eqref{mupn}) in the case $p(n)=n$. If $n -p(n) \to 0$ as $n \to \infty$, it is $\mu_1(p(n),n) + \frac{p(n)+1}{n+\nu}-(p(n)+1) \log (\frac{n+\nu}{2})
+  \log (n \, p(n))- \log 2)$. If $n-p(n) = c$ for a fixed $c \in \na$, 
the expectation is $\mu_2(p(n),n) + \frac{p(n)+1}{n+\nu} - (p(n)+1) \log (\frac{n+\nu}{2}) + \log (n \, p(n)) -\log 2$, with $\mu_2(p(n),n)$ defined in \eqref{munew}.
Finally for $n-p(n) =o(n)$, we obtain that
$\mu_3(p(n),n) + \frac{p(n)+1}{n+\nu} - (p(n)+1) \log (\frac{n+\nu}{2}) +  \log (n \, p(n)) -\log 2$, with $\mu_3(p(n),n)$ defined in \eqref{munew2}, is the correct expectation. 

\noindent
If $p(n)=p$ we obtain the mod-$\phi$ on $i \, \re$ result with the L\'evy exponent
$$
\eta(z) = (p+1) \bigl( \frac{z+1}{2} \bigr) \log \biggl(\frac{(z+1)(p+1)}{(z+1)p+1} \biggr) - \frac{pz}{2} \log 2 - \bigl( \frac{z+1}{2} \bigr) \log(z+1),
$$
expectation $\mu_4 = p \log 2$ and limiting function
$$
\psi^p(z) = (1+z)^{-\frac{p(p-1)}{4} - \frac p2 -(p+1) \bigl( \frac{\nu}{2} - \frac 12 \bigr)} \biggl( \frac{(z+1)(p+1)}{(z+1)p+1}\biggr)^{\frac{p(\nu-2)+\nu-1}{2}}.
$$
The case $\nu=0$ leads to the spherical model.

\end{theorem}

\begin{proof}
With \eqref{MellinSim2}, we first have to consider $(p(n)+1) \bigl(- \log \Gamma (m(n,\nu) +z/2) + \log \Gamma(m(n,\nu)) \bigr)$.
This term can be handled using Proposition \ref{Binetapp} with $m(n,\nu)= \frac{n +\nu}{2}$ and is equal to
$$
\frac z2 \, (p(n)+1)  \biggl( \frac{1}{n+\nu} - \log \bigl( \frac{n + \nu}{2} \bigr) \biggr) - \frac{p(n)+1}{2(n + \nu)} z^2 + {\mathcal O} \biggl( \frac{|z| + |z|^2 + |z|^3}{n} \biggr).
$$
Moreover, with Proposition \ref{Binetapp} we obtain
$$
\log \frac{ \Gamma \biggl( \frac{p(n)(n + \nu -2) + (n + \nu)}{2} + (p(n)+1) \frac z2 \biggr)}
{\Gamma \biggl( \frac{p(n)(n + \nu -2) + (n + \nu)}{2} + p(n) \frac z2 \biggr)}= \frac z 2 \bigl( \log (p(n) \, n) - \log 2 \bigr) + o(1).
$$
The results follow with \eqref{MellinSim2} and \eqref{MellinSim4} together with Theorem \ref{BSP}.
Finally, we have to consider the case $p(n)=p$.  With \eqref{Stirling}  we have
\begin{eqnarray*}
(p+1) \log \Gamma \bigl( \frac{n + \nu}{2} \bigr)  - (p+1) \log \Gamma \bigl( \frac{n + \nu}{2}  + \frac{nz}{2}\bigr)  & \sim &  n \frac z2 \bigl( (p+1) \bigl(1- \, \log \frac n2 \bigr)\bigr) \nonumber \\
& & \hspace{-7cm}+ (p+1)\, n \biggl( -  \bigl(\frac z2 + \frac{1}{2} \bigr) \log  \bigl(1+z \bigr) \biggr) 
-  (p+1) \biggl(  \frac{\nu}{2} - \frac 12 \biggr) \log \bigl( 1 +z \bigr). 
\end{eqnarray*}
Similarly we get
\begin{eqnarray*}
\log \Gamma \biggl( \frac{n}{2} (z+1)(p+1) + \frac{p(\nu-2) + \nu}{2} \biggr)  - \log \Gamma \biggl( \frac{n}{2} ((z+1) p+1) + \frac{p(\nu-2) + \nu}{2} \biggr)  & \sim &  \\
&& \hspace{-12cm}n \frac z2 \bigl( \log \frac n2 -1 \bigr) + \frac n2 (z+1)(p+1) \log \biggl( \frac{(z+1)(p+1)}{(z+1)p +1} \biggr)\\
&& \hspace{-12cm}+ \biggl( \frac{p(\nu-2) + \nu}{2} - \frac 12 \biggr) \log  \biggl( \frac{(z+1)(p+1)}{(z+1)p +1} \biggr).
\end{eqnarray*}
Hence the statement is proven.
\end{proof}

Interestingly enough, the Beta prime model (see \eqref{MellinPa3} and \eqref{MellinSim3}) behaves {\it differently!}
The reason for that is that in \eqref{MellinPa3}, the first summand is
$$
p(n) \biggl( \log \Gamma \bigl( \frac{\nu}{2} - \frac z2 \bigr) - \log \Gamma \bigl( \frac{\nu}{2} \bigr) \biggr) = : p(n) g(z),
$$
where $g(z)$ can be represented - for example - with the help of Binet's first formula \eqref{Binet1}. But if $p(n)$ depends on $n$,
this term is never part of the expectation of the log-volume of the random parallelotope. Moreover, $L(p(n),n-p(n), 1/2; z/2)$ leads to the parameter sequence 
$t_n = 2 \log n$, which does not correspond to the sequence $p(n)$.
Hence $g(z)$ is not part of the $\eta$-function in the definition of mod-$\phi$ convergence. With \eqref{MellinSim3} the same is true
for the log-volume of a random simplex in the Beta prime model. The only case where we expect mod-$\phi$ convergence is for a fixed $p$:

\begin{theorem}[Beta prime model, Parallelotope and Simplex]
Consider the sequence $n \bigl( \log (p! V P_{n,p}) - (p \log n -p) \bigr)$; it converges mod-$\phi$ on $i \, \re$ with
parameters $t_n= p \,n$ and limiting function
$$
\psi(z) = e^p (1+z)^{- \frac{p(p-1)}{4} - \frac p2}  \frac{\Gamma \bigl( \frac{\nu}{2} - \frac z2 \bigr)}{\Gamma \bigl( \frac{\nu}{2} \bigr)}.
$$
Here the L\'evy exponent is
\begin{equation} \label{levy}
\eta(z) = \frac 12 \log 2 + \bigl( \frac{z+1}{2} \bigr) \log \bigl( \frac{z+1}{2} \bigr).
\end{equation}
Next, consider the sequence $ n \bigl( \log ( p! V S_{n,p}) - (p \log n -p) \bigr)$; it converges mod-$\phi$ on $i \, \re$, with
parameters $t_n= p \,n$, limiting function
$$
\psi(z) = e^{p+1} (1+z)^{- \frac{p(p-1)}{4} - \frac p2} \frac{\Gamma \bigl( \frac{\nu}{2} - \frac z2 \bigr)}{\Gamma \bigl( \frac{\nu}{2} \bigr)}\, 
\frac{\Gamma \bigl( \frac{(p+1)\nu}{2} - p \frac z2 \bigr)}{\Gamma \bigl( \frac{(p+1)\nu}{2} - (p+1) \frac z2 \bigr)},
$$
and the same L\'evy exponent \eqref{levy}.
\end{theorem}

\begin{proof}
This is an immediate consequence of \eqref{MellinPa3} and \eqref{MellinSim3} combined with \eqref{Stir3}.
\end{proof}

We do not formulate the corresponding extended central limit theorems, precise deviations, large and moderate deviation principles 
and Berry-Esseen bounds for the log-volumes of random parallelotopes and simplices, because these results can be stated as in Theorems \ref{theoremmodL}, \ref{exCLT}, \ref{pLDP}, Corollary \ref{LDP} and Theorem \ref{BE}.

\begin{remark} The authors of the preprint arXiv:1708.00471v1 have been informed last year, that moment formulas
due to Mathai \cite{Mathai:1999} and Miles \cite{Miles:1971}, presented in the introduction, are examples of models with Gamma type moments and moreover behave as the Laguerre ensemble in random matrix theory. Knowing the work of \cite{Borgoetal:2017} before it, we decided
to collect our observations after the posting of \cite{Borgoetal:2017}. The results in this subsection include the results
of arXiv:1708.00471v1. 
Moreover, the spherical model (uniform Gram ensembles)
has already been considered in \cite{Rouault:2007} and \cite{Borgoetal:2017}. 
\end{remark}

\section{Appendix}
In the appendix we are collecting some well known facts about the Gamma function and the Barnes $G$ function.
For $z \in \cn$ with $\operatorname{Re}(z) >0$, the complex Gamma function is given by
$$
\Gamma(z) = \int_0^{\infty} e^{-t} t^{z-1} \, dt.
$$
The first Binet's formula is saying that 
\begin{equation} \label{Binet1.2}
\log \Gamma(z) = \bigl( z - \frac 12 \bigr) \log z - z +1 + \int_0^{\infty} \varphi(s) (e^{-sz} - e^{-s}) \, ds, \quad \operatorname{Re}(z) >0,
\end{equation}
see \cite[p.242]{Whittaker/Watson:1996}.
Here, the function $\varphi$ is given by $\varphi(s) = \bigl( \frac 12 - \frac 1s + \frac{1}{e^s-1} ) \frac 1s$ and
for every $s \geq 0$ satisfies $0 < \varphi(s) \leq \lim_{s \to 0} \varphi(s) = \frac{1}{12}$. 
The second Binet's formula is
\begin{equation} \label{Binet2}
\log \Gamma(z) = \bigl( z - \frac 12 \bigr) \log z - z + \frac 12 \log (2 \pi) + 2 \int_0^{\infty} \frac{\arctan (s/z)}{e^{2 \pi s} -1} \, ds, \quad \operatorname{Re}(z) >0,
\end{equation}
where $\arctan y := \int_0^y \frac{dt}{1+t^2}$ for any complex $y$, with integration along a straight line, see \cite[p.243]{Whittaker/Watson:1996}.

The derivative of the logarithm of the Gamma function $\psi(z)$, called the Digamma function, can be represented, differentiating \eqref{Binet1.2}, as
\begin{equation} \label{digamma}
\Psi(z) = \log z - \int_0^{\infty} e^{-sz} \bigl( s \varphi(s)  + \frac 12 \bigr) \, ds
\end{equation}
and
\begin{equation} \label{boundphi}
0 < s \varphi(s) + \frac 12 < 1.
\end{equation}

The Barnes $G$ function is defined as the solution of
$$
G(z+1) = G(z) \Gamma(z).
$$
Its derivative satisfies (see \cite[p.258]{Whittaker/Watson:1996})
\begin{equation} \label{logG}
\frac{G'(z)}{G(z)} = (z-1) \Psi(z) - z + \frac 12 \log (2 \pi) + \frac 12.
\end{equation}
It is known that
\begin{equation} \label{Barnes1}
\log G(z+1) = \frac{z(z-1)}{2} + \frac{z}{2} \log (2 \pi) + z \log \Gamma(z) - \int_0^z \log \Gamma(x) \, dx, \quad \operatorname{Re}(z) >0,
\end{equation}
see \cite{Barnes:1900}.
If we put the second Binet's formula \eqref{Binet2} into \eqref{Barnes1} we obtain for all $z \in \cn$ with $\operatorname{Re}(z) >0$
\begin{eqnarray*} \label{Barnes2}
\log G(z+1) & = &\frac{z}{2} - \frac{3 z^2}{2} + z \log (2 \pi) + z \biggl( z - \frac 12 \biggr) \log z - \frac{z}{2} \biggl( 1 - \frac{z}{2} + (z-1) \log z \biggr) \\
&& + \frac{z^2}{2} - \frac{z}{2} \log(2 \pi) - 2 \int_0^{\infty} \frac{ \frac 12 \log(z^2+s^2) - s \log s}{e^{2 \pi s} -1} \, ds,
\end{eqnarray*}
which gives
\begin{equation} \label{Barnes2}
\log G(z+1) = \frac{z^2}{2} \log z - \frac 34 z^2 + \frac{z}{2} \log (2 \pi) - \int_0^{\infty} \log (1 + z^2 s^{-2}) \frac{ s \, ds}{e^{2 \pi s} -1}.
\end{equation}

We apply the famous Abel-Plana summation formula, see \cite[p.290]{Olver:1997}:

\begin{theorem} \label{Abel-Plana}
Let $f$ be a holomorphic function on the strip $\{ z \in \cn: 0 \leq \operatorname{Re}(z) \leq n \}$. Suppose that
$f(z) = o \bigl( e^{2 \pi |\operatorname{Im}(z)|} \bigr)$ as $\operatorname{Im}(z) \to \pm \infty$, uniformly with respect to
$\operatorname{Re}(z) \in [0,n]$. Then
\begin{eqnarray*}
\sum_{k=0}^{n-1} f(k)& =& \int_0^n f(s) \, ds + \frac 12 f(0) - \frac 12 f(n) + i \int_0^{\infty} \frac{f(is)- f(-is)}{e^{2 \pi s} -1} \, ds\\
&& - i \int_0^{\infty} \frac{f(n + is) - f(n -is)}{e^{2 \pi s} -1} \, ds.
\end{eqnarray*}
\end{theorem}

\medskip

\noindent
{\bf Acknowledgement.} Lukas Knichel has been supported by the German Research Foundation (DFG) via Research Training Group RTG 2131 {\it High dimensional phenomena in probability- fluctuations and discontinuity}. The authors like to thank Martina Dal Borgo for presenting and discussing her work \cite{Borgoetal:2017}
at a summer-school in Ghiffa/Italy in September 2016, organised by the first author. 

\newcommand{\SortNoop}[1]{}\def\cprime{$'$} \def\cprime{$'$}
  \def\polhk#1{\setbox0=\hbox{#1}{\ooalign{\hidewidth
  \lower1.5ex\hbox{`}\hidewidth\crcr\unhbox0}}} \def\cprime{$'$}
\providecommand{\bysame}{\leavevmode\hbox to3em{\hrulefill}\thinspace}
\providecommand{\MR}{\relax\ifhmode\unskip\space\fi MR }
% \MRhref is called by the amsart/book/proc definition of \MR.
\providecommand{\MRhref}[2]{%
  \href{http://www.ams.org/mathscinet-getitem?mr=#1}{#2}
}
\providecommand{\href}[2]{#2}

%\bibliographystyle{amsplain}
%\bibliography{1.03.2017}
\end{document}